\documentclass[11pt]{article}
\usepackage{smile}

\usepackage{fullpage}
\usepackage{url}
\usepackage{lscape}
\usepackage{bigints}
\usepackage{framed}
\usepackage{mdframed}
\usepackage{enumerate}
\usepackage[inline]{enumitem}
\usepackage[T1]{fontenc}
\usepackage{moresize}
\usepackage{bm}
\usepackage{bbm}
\usepackage{dsfont}
\usepackage{amsmath}
\usepackage{amssymb}
\usepackage{amsthm}
\usepackage{amsfonts}
\usepackage{stmaryrd}
\usepackage{array}
\usepackage{mathrsfs}
\usepackage{mathtools} 
\usepackage{extarrows}
\usepackage{stackrel}
\usepackage{relsize,exscale}
\usepackage{scalerel}
\usepackage{colortbl}
\usepackage[nodisplayskipstretch]{setspace}
\usepackage{color}
\usepackage[usenames,dvipsnames]{xcolor}
\usepackage{cancel}
\usepackage{soul}
\usepackage{undertilde}
\usepackage{xfrac}
\usepackage{siunitx}
\usepackage{graphicx}
\usepackage{float}
\usepackage{rotating}
\usepackage{subcaption}
\usepackage{overpic}
\usepackage[all]{xy}
\usepackage{tikz}
\usetikzlibrary{arrows,matrix,positioning,calc,automata,patterns}
\usepackage{booktabs}
\usepackage{dcolumn}
\usepackage{multirow}
\usepackage{diagbox}
\usepackage{tabularx}
\usepackage{verbatim}
\usepackage{listings}
\usepackage[ruled,vlined]{algorithm2e}
\usepackage{fancyvrb}
\usepackage{hyperref}
\usepackage[round]{natbib}
\usepackage{sectsty}

\hypersetup{
    bookmarks=true,         
    unicode=false,          
    pdftoolbar=true,        
    pdfmenubar=true,        
    pdffitwindow=false,     
    pdfstartview={FitH},    
    pdftitle={My title},    
    pdfauthor={Author},     
    pdfsubject={Subject},   
    pdfcreator={Creator},   
    pdfproducer={Producer}, 
    pdfkeywords={key1, key2}, 
    pdfnewwindow=true,      
    colorlinks=true,        
    linkcolor=blue,         
    citecolor=blue,         
    filecolor=blue,         
    urlcolor=cyan           
}

\usepackage{stackengine}
\stackMath
\newcommand\tenq[2][1]{%
\def\useanchorwidth{T}%
\ifnum#1>1%
\stackunder[0pt]{\tenq[\numexpr#1-1\relax]{#2}}{\!\scriptscriptstyle\thicksim}%
\else%
\stackunder[1pt]{#2}{\!\scriptstyle\thicksim}%
\fi%
}

\makeatletter
\DeclareRobustCommand\widecheck[1]{{\mathpalette\@widecheck{#1}}}
\def\@widecheck#1#2{%
    \setbox\z@\hbox{\m@th$#1#2$}%
    \setbox\tw@\hbox{\m@th$#1%
       \widehat{%
          \vrule\@width\z@\@height\ht\z@
          \vrule\@height\z@\@width\wd\z@}$}%
    \dp\tw@-\ht\z@
    \@tempdima\ht\z@ \advance\@tempdima2\ht\tw@ \divide\@tempdima\thr@@
    \setbox\tw@\hbox{%
       \raise\@tempdima\hbox{\scalebox{1}[-1]{\lower\@tempdima\box
\tw@}}}%
    {\ooalign{\box\tw@ \cr \box\z@}}}
\makeatother

\def\tr{\mathop{\text{tr}}\kern.2ex}

\def\P{{\mathrm P}}
\def\Q{{\mathrm Q}}

\def\E{{\mathrm E}}

\def\d{{\mathrm d}}

\def\Pn{{\mathbb{P}_{n}}}
\def\Pk{{\mathbb{P}_{n,k}}}
\def\BtPk{{\mathbb{P}_{n,k}^{\ast}}}
\def\tBtPk{{\tilde{\mathbb{P}}_{n,k}^{\ast}}}
\def\Gk{{\mathbb{G}_{n,k}}}
\def\BtGk{{\mathbb{G}_{n,k}^{\ast}}}
\def\BtGko{{\mathbb{G}_{n,k}^{\circ}}}

\newcommand{\zahl}[1]{\llbracket #1\rrbracket}
\newcommand\yestag{\addtocounter{equation}{1}\tag{\theequation}}

\newcolumntype{L}[1]{>{\raggedright\let\newline\\\arraybackslash\hspace{0pt}}m{#1}}
\newcolumntype{C}[1]{>{  \centering\let\newline\\\arraybackslash\hspace{0pt}}m{#1}}
\newcolumntype{R}[1]{>{ \raggedleft\let\newline\\\arraybackslash\hspace{0pt}}m{#1}}
\newcolumntype{d}[1]{D{.}{.}{#1}}
\newcolumntype{H}{>{\setbox0=\hbox\bgroup}c<{\egroup}@{}}
\newcolumntype{Z}{>{\setbox0=\hbox\bgroup}c<{\egroup}@{\hspace*{-\tabcolsep}}}
\newcolumntype{b}{X}
\newcolumntype{s}{>{\hsize=.5\hsize}X}

\DeclarePairedDelimiter\floor{\lfloor}{\rfloor}

\newcommand{\plim}{\text{plim}}

\numberwithin{equation}{section}

\newtheorem{theorem}{Theorem}[section]
\newtheorem{lemma}{Lemma}[section]
\newtheorem{proposition}{Proposition}[section]
\newtheorem{assumption}{Assumption}[section]

\newtheorem{innercustomgeneric}{\customgenericname}
\providecommand{\customgenericname}{}
\newcommand{\newcustomtheorem}[2]{%
  \newenvironment{#1}[1]
  {%
   \renewcommand\customgenericname{#2}%
   \renewcommand\theinnercustomgeneric{##1}%
   \innercustomgeneric
  }
  {\endinnercustomgeneric}
}
\newcustomtheorem{customdefinition}{Definition}
\newcustomtheorem{customdefinitions}{Definitions}
\newcustomtheorem{customtheorem}{Theorem}
\newcustomtheorem{customassumption}{Assumption}
\newcustomtheorem{customlemma}{Lemma}
\newcustomtheorem{customexample}{Example}
\theoremstyle{definition}

\newtheorem{example}{Example}[section]
\newtheorem{remark}{Remark}[section]

\usepackage{enumitem}
\makeatletter
\newcommand{\mylabel}[2]{#2\def\@currentlabel{#2}\label{#1}}
\makeatother

\graphicspath{{./fig3/}}

\allowdisplaybreaks

\begin{document}

\setlength{\abovedisplayskip}{5pt}
\setlength{\belowdisplayskip}{5pt}
\setlength{\abovedisplayshortskip}{5pt}
\setlength{\belowdisplayshortskip}{5pt}
\hypersetup{colorlinks,breaklinks,urlcolor=blue,linkcolor=blue}

\title{\LARGE Bootstrap consistency for general double/debiased machine learning estimators}

\author{Ziming Lin\thanks{Department of Statistics, University of Washington, Seattle, WA 98195, USA; e-mail: {\tt zmlin@uw.edu}} ~~and~~Fang Han\thanks{Department of Statistics, University of Washington, Seattle, WA 98195, USA; e-mail: {\tt fanghan@uw.edu}}
}

\date{\today}

\maketitle

\vspace{-1em}

\begin{abstract}
Double/debiased machine learning (DML) provides a general framework for inference with high-dimensional or otherwise complex nuisance parameters by combining Neyman-orthogonal scores with cross-fitting, thereby circumventing classical Donsker-type conditions in many modern machine-learning settings. Despite its strong empirical performance, bootstrap inference for DML estimators has received little theoretical justification. This is particularly noteworthy since bootstrap methods are suggested and used for inference on DML estimators, even though bootstrap procedures can fail for estimators that are root-\(n\) consistent and asymptotically normal. This paper fills this gap by establishing bootstrap validity for DML estimators under general exchangeably weighted resampling schemes, with Efron’s bootstrap as a special case. Under exactly the same conditions required for the validity of DML itself, we prove that the bootstrap law converges conditionally weakly to the sampling law of the original estimator.
\end{abstract}


\section{Introduction}

Double/debiased machine learning (DML) \citep{chernozhukov2018double} provides a general framework for conducting inference on a low-dimensional target parameter in the presence of high-dimensional or otherwise complex nuisance components. A key advantage of DML is that it avoids the classical Donsker-type conditions that are ubiquitous in semiparametric inference. This is accomplished by combining Neyman-orthogonal score functions with cross-fitting. As a result, DML offers a convenient and broadly applicable framework for constructing root-\(n\) consistent and asymptotically normal estimators under conditions well suited to modern machine-learning methods for nuisance estimation in high-complexity settings.

While normal approximations and Wald-type inference are available and computationally appealing for inferring DML estimators, bootstrap procedures \citep{efron1979bootstrap, wu1986jackknife} remain attractive for several reasons. They provide an intuitive route to distributional approximation \citep{imbens2024causal}, interface naturally with black-box nuisance learners, and readily accommodate a variety of inferential methods, including percentile, basic, and studentized confidence intervals \citep{hall1988theoretical, diciccio1988review}. It is also worth noting that, despite the current lack of general theoretical justification, bootstrap methods have already been suggested and used for inference on DML and related estimators \citep{cai2020nonparametric,fingerhut2022coordinated, dukes2024doubly}.

Our interest in the bootstrap validity of DML estimators is motivated by two considerations. First, it is well known that root-\(n\) consistency and asymptotic normality alone do not guarantee bootstrap validity. This point has become especially clear in recent work, including \cite{abadie2008failure} and \cite{lin2023failure} (see also \cite{lin2026consistency}), which establish bootstrap inconsistency for nearest-neighbor matching estimators and Chatterjee's rank correlation. These estimators, much like DML estimators, are root-\(n\) consistent and asymptotically normal, yet the bootstrap is invalid.

Second, from a technical standpoint, bootstrap inference for estimators constructed via sample-splitting and cross-fitting, as in the DML framework, remains largely unexplored and is far from straightforward. In particular, even when bootstrap validity does hold, it remains unclear whether it can be established under exactly the same conditions that justify Wald-type inference for the original DML estimator.


This paper studies the bootstrap validity of DML estimators under the general exchangeably weighted bootstrap scheme of \citet{praestgaard1993exchangeably}. To highlight the main ideas, we focus on the simplest---and arguably most commonly used---bootstrap implementation, namely, a single round of resampling from the full dataset. By contrast, bootstrap procedures that resample separately within folds are technically easier to analyze and lead to the same conclusion. On the theoretical side, under exactly the same conditions imposed in \citet{chernozhukov2018double} for the validity of the original DML estimators, we establish a bootstrap linear representation and prove conditional weak convergence of the bootstrap law to the sampling law of the original estimator.

Our contribution is best understood in relation to the literature on exchangeably weighted bootstrap for semiparametric \(M\)-estimation \citep{cheng2010bootstrap}, as well as more recent bootstrap theory for asymptotically linear estimators with data-adaptive nuisance estimation \citep{tang2024consistency}. The present paper complements this line of work by treating the orthogonal, cross-fitted DML setting directly. In particular, although our bootstrap procedure is still exchangeably weighted, the analysis is carried out under the orthogonality and cross-fitting conditions standard in DML, rather than under the classical Donsker-type empirical-process assumptions commonly invoked in related bootstrap studies.


The remainder of the paper is organized as follows.
Section~\ref{sec:setup} reviews the DML framework and introduces the
exchangeably weighted bootstrap procedure.
Section~\ref{sec:main} presents the main bootstrap consistency theorem.
Section~\ref{sec:proof} contains the proofs and additional technical lemmas.

\paragraph*{Notation.}
For any integers $n,d\ge 1$, let $\zahl{n}:= \{1,2,\ldots,n\}$, and $\bR^d$ be the $d$-dimensional real space.
For any $a,b \in \bR$, write $a \vee b = \max\{a,b\}$ and $a \wedge b = \min\{a,b\}$. We use $\lVert \cdot \rVert$ to denote the Euclidean norm.
For any two real sequences $\{a_n\}_{n\ge1}$ and $\{b_n\}_{n\ge1}$, we write $a_n \lesssim b_n$ (or $a_n=O(b_n)$) if there exists a universal constant that upper bounds $|a_n|/|b_n|$ for all sufficiently large $n$, and write $a_n=o(b_n)$ if $a_n/b_n \to 0$ as $n$ goes to infinity. We write $a_n\asymp b_n$ if both $a_n\lesssim b_n$ and $b_n\lesssim a_n$ hold true. For any sequence of random elements $\{Z_n\}_{n \ge 1}$, write $Z_n = o_\P(1)$ if $Z_n$ goes to $0$ in $\P$-probability and $Z_n = O_\P(1)$ if $Z_n$ is bounded in $\P$-probability. We use ``$\plim$'' to denote the limit in probability for random variables. 

\section{Setup}\label{sec:setup}

\subsection{Double/debiased machine learning framework}

The DML framework assumes the existence of a true low-dimensional target parameter \(\theta_0 \in \Theta \subseteq \bR^{d_\theta}\) and a true nuisance parameter \(\eta_0\) belonging to a convex set \(T\) in a normed vector space equipped with norm \(\lVert \cdot \rVert_T\). The pair \((\theta_0,\eta_0)\) is assumed to be identified by the moment condition
\begin{align*}
    \P_X\big[\psi(\cdot;\theta_0,\eta_0)\big] = 0. \yestag\label{eq: moment condition}
\end{align*}
Here \(\psi=(\psi_1,\ldots,\psi_{d_\theta})^\top\) is a known vector-valued score function,  \(X\) is a random variable with the law $\P_X$, 
and we use the customary operator notation
\[
\P_X[f] := \int f\, d\P_X,
\]
whenever the right-hand side is well defined.

Given \(n\) observations \(\{X_i\}_{i=1}^n\) drawn from \(\P_X\), let \(K\) be a fixed positive integer. For notational and proof simplicity, we assume that \(n\) is divisible by \(K\), and the index set \(\zahl{n}\) is partitioned into \(K\) nonrandom folds \(\{I_k\}_{k=1}^K\), each of cardinality \(m := n/K\). For each \(k \in \zahl{K}\), let \(\hat{\eta}_{0,k}\) denote a user-specified nuisance estimator constructed using only the observations outside fold \(I_k\), that is,
\[
\hat{\eta}_{0,k}
=
\hat{\eta}_{0,k}\Big(\Big\{X_i\Big\}_{i \in I_k^c}\Big),
\qquad k \in \zahl{K},
\]
where \(I_k^c := \zahl{n} \setminus I_k\).

For each \(k \in \zahl{K}\), let \(\Pk\) denote the empirical measure based on the observations in the \(k\)-th fold, namely,
\[ 
\Pk := \frac{1}{m} \sum_{i\in I_k} \delta_{X_i} 
\] with $\delta_x$ representing the Dirac measure at the point mass $x$. The fold-specific estimator \(\check{\theta}_{0,k}\) is then defined as an approximate \(\epsilon_n\)-solution to the empirical moment condition:
\begin{align}\label{eq:DML}
    \Big\lVert \Pk\Big[\psi(\cdot;\check{\theta}_{0,k},\hat{\eta}_{0,k})\Big]\Big\rVert
    \le
    \inf_{\theta\in\Theta}
    \Big\lVert \Pk\Big[\psi(\cdot;\theta,\hat{\eta}_{0,k})\Big]\Big\rVert
    + \epsilon_n,
    \qquad \text{for all } k\in\zahl{K}.
\end{align}
Here \(\epsilon_n\) represents the allowed computational error.

Finally, the cross-fitted DML estimator \(\check{\theta}_0\) of \(\theta_0\) is defined as the average of the fold-specific estimators:
\[
\check{\theta}_0 := \frac{1}{K}\sum_{k\in\zahl{K}} \check{\theta}_{0,k}.
\]

\subsection{Bootstrap procedure}

For the bootstrap procedure, we consider the class of exchangeably weighted bootstrap as introduced in \cite{praestgaard1993exchangeably}, which gives general conditions for the bootstrap weights.

\begin{assumption}
\label{asp: BTWeights}
For the triangular array of bootstrap weights $W$ of the law $\P_w$, it is assumed that the following holds true.
\begin{enumerate} [itemsep=-.5ex, label=(\roman*)]
    \item The random vector $W=(W_1, \ldots, W_n)^{\top}$ is exchangeable.
    \item $W_i \ge 0$ for each $i \in \zahl{n}$ and $\sum_{i=1}^n W_i = n$.
    \item $\limsup_{n \to \infty} \int_0^{\infty} \sqrt{\P_W \left(W_1 \ge t\right)} \d t \le C$ for some constant $C< \infty$.
    \item $\lim_{t_0 \to \infty} \limsup_{n \to \infty} \sup_{t \ge t_0} t^2 \P_W \left(W_1 \ge t\right) = 0$.
    \item $\operatorname{plim}_{n \to \infty} n^{-1} \sum_{i=1}^n (W_i-1)^2 = c^2 >0$.
\end{enumerate}
\end{assumption}

Assumption \ref{asp: BTWeights} covers standard exchangeably weighted resampling schemes and (iii)-(v) can be implied by moment conditions on $W$; cf. Section 3 of \cite{praestgaard1993exchangeably}. 

\begin{example}\label{exp: BT, class}
 The resampling schemes satisfying Assumption~\ref{asp: BTWeights} include the following.
\begin{enumerate} [itemsep=-.5ex, label=(\roman*)]
    \item Efron’s bootstrap \citep{efron1979bootstrap}: the weight vector is multinomially distributed such that
    \[
    W \sim {\rm Mult}(n;(n^{-1}, \ldots, n^{-1})).
    \]
    In this case, $c^2=1$.
    \item Normalized multiplier bootstrap: the weight vector is 
    \[
    W_i=Y_i\Big/\Big(\frac{1}{n}\sum_{i=1}^n Y_i\Big), ~~\text{for }i\in\zahl{n},
    \]
    where $\{Y_i\}_{i=1}^n$ are independent and identically distributed positive random variables satisfying 
    \[
    \int_0^{\infty} \sqrt{\P_Y \left(Y_1 \ge t\right)} \d t<\infty.
    \]
    In this case $c^2 = \Var[Y_1]/(\E[Y_1])^2$.  
    
    A special normalized multiplier bootstrap procedure is the Bayesian bootstrap \citep{rubin1981bayesian}, obtained by taking $Y_1 \sim {\rm Exp}(1)$, for which $c^2=1$. More generally, $Y_1 \sim \Gamma(\alpha,\alpha)$ yields $c^2 = 1/\alpha$. 
    
    \item Double bootstrap \citep{beran1987prepivoting}: let 
    \[
    M=(M_1,\ldots,M_n)^\top \sim {\rm Mult}(n;(n^{-1}, \ldots, n^{-1})); 
    \]
    then conditioning on $M$, $W$ is sampled from 
    \[
    W|M \sim {\rm Mult}(n;(n^{-1}M_1, \ldots, n^{-1}M_n)). 
    \]
    In this case $c^2=2$.

\end{enumerate}
\end{example}

We are now ready to define the bootstrapped DML estimator, $\check{\theta}_0^*$. Without refitting the nuisance estimator $\hat{\eta}_{0,k}$, we define the {\it fold-wise} bootstrapped estimator $\check{\theta}_{0,k}^*$ to be {\it any} vector that satisfies the following approximate $\epsilon_n$-solution condition:
\begin{align}\label{eq:boots-est}
   \Big\lVert \BtPk \Big[\psi(\cdot; \check{\theta}_{0,k}^*, \hat{\eta}_{0,k})\Big] \Big\rVert \le \inf_{\theta \in \Theta} \Big\lVert \BtPk \Big[\psi(\cdot; \theta, \hat{\eta}_{0,k})\Big] \Big\rVert + \epsilon_n, 
\end{align}
where $\BtPk$ is the {\it fold-restricted weighted empirical measure} so that 
\[
\BtPk[f] = \frac{1}{m}\sum_{i\in I_k}W_i f(X_i)  ~~\text{ for any measurable function $f$.}
\]
Of note, in practice, there may be multiple solutions satisfying \eqref{eq:boots-est}. In particular, although this occurs with vanishing probability, it is possible that \(W_i = 0\) for all \(i \in I_k\) for some \(k \in \zahl{K}\). In such cases, \(\check{\theta}_{0,k}^*\) may be taken as any vector.

The bootstrapped DML estimator is then the average of the fold-wise estimators: 
\begin{align*}
    \check{\theta}_0^* = \frac{1}{K} \sum_{k \in \zahl{K}} \check{\theta}_{0,k}^*.
\end{align*}

\begin{remark}
In the special case of Efron’s bootstrap, the weighted empirical measure \(\BtPk\) admits the following equivalent representation. First, draw bootstrap indices \(\{j_i\}_{i=1}^n\) independently from the uniform distribution on \(\zahl{n}\). Then form the bootstrap sample \(\{X_i^* := X_{j_i}\}_{i=1}^n\). Under this representation, the fold-restricted bootstrap empirical measure can be written as
\begin{align*}
    \BtPk[f]
    =
    \frac{1}{m}\sum_{i=1}^n \ind(j_i \in I_k)\, f(X_i^*),
    \qquad \text{for any measurable function } f.
\end{align*}
\end{remark}

\begin{remark}
For simplicity of presentation, this paper focuses exclusively on bootstrap procedures based on a single round of resampling from the full dataset. In practice, however, it may be also natural to consider within-fold resampling. Specifically, for each \(k \in \zahl{K}\), let \((W_{k,1},\ldots,W_{k,m})\) be bootstrap weights satisfying Assumption \ref{asp: BTWeights}, generated independently across folds. The corresponding weighted bootstrap empirical measure is then defined by
\[
\tBtPk[f]
=
\frac{1}{m}\sum_{i\in I_k} W_{k,i} f(X_i),
\qquad \text{for any measurable function } f.
\]
The associated bootstrapped DML estimator can be defined analogously. By inspecting the proofs below, it is clear that the bootstrap consistency result established in Theorem \ref{thm:BT,DML,consistency} continues to hold for this within-fold bootstrap procedure as well. We omit the details for brevity.
\end{remark}

\section{Theory}\label{sec:main}

To present our main theory, we first recall the normal approximation result for the original DML estimator from \citet{chernozhukov2018double}. We begin with some additional notation. For any measurable real-valued function \(f\), let
\[
\lVert f \rVert_{\P,q}
=
\Big(\int |f(x)|^q \, d\P(x)\Big)^{1/q}
\]
denote its \(L^q(\P)\) norm. For simplicity, whenever no confusion can arise, we suppress the distinction between \(\P_X\) and \(\P_W\) and write both simply as \(\P\).

We also define, for any \(\eta \in T\), the G\^ateaux derivative of the map \(\eta \mapsto \P[\psi(\cdot;\theta_0,\eta)]\) at \(\eta_0\) in the direction \(\eta-\eta_0\) (see, e.g., Chapter 7.2 of \cite{luenberger1997optimization}) by
\begin{align*}
\partial_\eta \P\Big[\psi\Big(\cdot;\theta_0,\eta_0\Big)\Big][\eta-\eta_0]
:=
\partial_r \P\Big[\psi\Big(\cdot;\theta_0,\eta_0 + r(\eta-\eta_0)\Big)\Big]\Big\rvert_{r=0},
\end{align*}
and we assume throughout that this derivative exists. 

Hereafter, let \(c_0>0\), \(c_1>0\), \(a>1\), \(v>0\), and \(q>2\) be some fixed finite constants, and let \(\{\delta_n\}_{n\ge 1}\) and \(\{\tau_n\}_{n\ge 1}\) be some sequences of positive constants converging to zero. We further assume that these sequences satisfy
\[
\delta_n \ge n^{-1/2+1/q}\log n,
\qquad \text{and} \qquad
n^{-1/2}\log n \le \tau_n \le \delta_n,
\qquad \text{for all } n\ge 1.
\]

Let \(\{\cP_n\}_{n\ge 1}\) denote the sequence of classes of data-generating distributions over which all subsequent results are stated uniformly. We now summarize the assumptions needed to establish the theoretical properties of DML. The first assumption collects the conditions introduced in Section~\ref{sec:setup}.

\begin{assumption} \label{asp: BT,DML,data}
For \(\P \in \cP_n\), the following conditions are assumed to hold:
\begin{enumerate}[itemsep=-.5ex,label=(\roman*)]
    \item \(\{X_i\}_{i=1}^n\) are independent draws from \(\P\);
    \item there exists a fixed integer \(K \ge 2\) such that \(n\) is divisible by \(K\) and \(m = n/K\);
    \item the score function \(\psi(\cdot;\theta_0,\eta_0)\) satisfies the moment condition
    \[
    \P\big[\psi(\cdot;\theta_0,\eta_0)\big] = 0;
    \]
    \item the computational error \(\epsilon_n\) in \eqref{eq:DML} and \eqref{eq:boots-est} satisfies
    \[
    \epsilon_n = o(\delta_n n^{-1/2}).
    \]
\end{enumerate}
\end{assumption}


We next impose a collection of regularity conditions paralleling Assumption 3.3 of \citet{chernozhukov2018double}. They formalize the local moment structure of the problem and the approximate Neyman-orthogonality of the score.

\begin{assumption}\label{asp: BT,DML,1}
For every \(n \ge 3\) and every \(\P \in \mathcal{P}_n\), the following conditions hold true.
\begin{enumerate}[itemsep=-.5ex,label=(\roman*)]
    \item The true parameter \(\theta_0\) satisfies the moment restriction \eqref{eq: moment condition}, and the parameter space \(\Theta\) contains the ball centered at \(\theta_0\) with radius \(c_1 n^{-1/2}\log n\).

    \item The mapping
    \[
    (\theta,\eta) \mapsto \P\big[\psi(\cdot;\theta,\eta)\big]
    \]
    is twice continuously G\^ateaux differentiable on \(\Theta \times T\).

    \item For every \(\theta \in \Theta\), the following local identification condition is satisfied:
    \[
    2\Big\lVert \P\big[\psi(\cdot;\theta,\eta_0)\big] \Big\rVert
    \ge
    \|J_0(\theta-\theta_0)\| \wedge c_0,
    \]
    where
    \begin{align*}
        J_0
        :=
        \left.
        \partial_{\theta^\top}
        \Big\{
        \P\big[\psi(\cdot;\theta,\eta_0)\big]
        \Big\}
        \right|_{\theta=\theta_0}.
    \end{align*}
    In addition, all singular values of \(J_0\) lie in the interval \([c_0,c_1]\).

    \item The score function \(\psi\) satisfies an approximate Neyman-orthogonality condition on the nuisance realization set \(\mathcal{T}_n \subset T\) with tolerance level \(\lambda_n = \delta_n n^{-1/2}\). Specifically, for every \(\eta \in \mathcal{T}_n\),
    \begin{align*}
        \Big\lVert
        \partial_\eta \P\big[\psi(\cdot;\theta_0,\eta_0)\big][\eta-\eta_0]
        \Big\rVert
        \le
        \lambda_n.
    \end{align*}
\end{enumerate}
\end{assumption}

We next impose a collection of regularity conditions paralleling Assumption 3.4 of \cite{chernozhukov2018double}. These conditions govern the complexity and regularity of the score, as well as the quality of nuisance estimation.

\begin{assumption}\label{asp: BT,DML,2}
For every \(n \ge 3\) and every \(\P \in \mathcal{P}_n\), the following conditions are assumed to be satisfied.
\begin{enumerate}[itemsep=-.5ex,label=(\roman*)]
    \item Let \(I\) be a random subset of \(\zahl{n}\) of cardinality \(n/K\). The nuisance estimator
    \[
    \hat{\eta}_0 = \hat{\eta}_0\big(\{X_i\}_{i\in I^c}\big)
    \]
    belongs to the realization set \(\mathcal{T}_n\) with \(\P\)-probability at least \(1-\Delta_n\), where \(\{\Delta_n\}_{n\ge 1}\) is a sequence of positive constants converging to zero. Moreover, \(\mathcal{T}_n\) contains \(\eta_0\) and satisfies the conditions stated below.

    \item The parameter space \(\Theta\) is bounded. For each \(\eta \in \mathcal{T}_n\), define the class of score functions
    \[
    \mathcal{F}_{1,\eta}
    :=
    \Big\{
    \psi_j(\cdot;\theta,\eta)
    :
    j=1,\ldots,d_\theta,\ \theta\in\Theta
    \Big\}.
    \]
    This class is assumed to be suitably measurable, and its uniform covering entropy satisfies
    \begin{align*}
        \sup_{\Q}
        \log N\Big(
        \epsilon \|F_{1,\eta}\|_{\Q,2},
        \mathcal{F}_{1,\eta},
        \|\cdot\|_{\Q,2}
        \Big)
        \le
        v \log(a/\epsilon),
        \qquad \text{for all } 0<\epsilon\le 1,
    \end{align*}
    where \(F_{1,\eta}\) is a measurable envelope of \(\mathcal{F}_{1,\eta}\) satisfying
    \[
    \|F_{1,\eta}\|_{\P,q}\le c_1.
    \]

    \item The following rate conditions hold for the quantities \(r_n\), \(r_n'\), and \(\lambda_n'\):
    \begin{align*}
        r_n
        &:=
        \sup_{\eta\in\mathcal{T}_n,\ \theta\in\Theta}
        \Big\lVert
        \P\big[\psi(\cdot;\theta,\eta)-\psi(\cdot;\theta,\eta_0)\big]
        \Big\rVert
        \le
        \delta_n \tau_n, \\
        r_n'
        &:=
        \sup_{\eta\in\mathcal{T}_n,\ \|\theta-\theta_0\|\le \tau_n}
        \Big(
        \P \big\|
        \psi(\cdot;\theta,\eta)-\psi(\cdot;\theta_0,\eta_0)
        \big\|^2
        \Big)^{1/2},
        \qquad
        r_n' \log^{1/2}(1/r_n') \le \delta_n, \\
        \lambda_n'
        &:=
        \sup_{\substack{r\in(0,1),\ \eta\in\mathcal{T}_n,\\ \|\theta-\theta_0\|\le \tau_n}}
        \Big\lVert
        \partial_r^2
        \P\Big[
        \psi\big(\cdot;\theta_0+r(\theta-\theta_0),\,\eta_0+r(\eta-\eta_0)\big)
        \Big]
        \Big\rVert
        \le
        \delta_n n^{-1/2}.
    \end{align*}

    \item The score has nondegenerate covariance at the truth, in the sense that all eigenvalues of
    \begin{align*}
        \P\Big[
        \psi(\cdot;\theta_0,\eta_0)\psi(\cdot;\theta_0,\eta_0)^\top
        \Big]
    \end{align*}
    are bounded below by \(c_0\).
\end{enumerate}
\end{assumption}

Under these conditions, the original DML estimator \(\check{\theta}_0\) admits the following normal approximation result, which we again summarize from \cite{chernozhukov2018double}.

\begin{theorem}[Theorem 3.3, \cite{chernozhukov2018double}]
\label{thm:DML,Gaussian}
Suppose Assumptions~\ref{asp: BT,DML,data}--\ref{asp: BT,DML,2} hold. Then the DML estimator \(\check{\theta}_0\) lies in an \(n^{-1/2}\)-neighborhood of \(\theta_0\) and admits an asymptotically linear expansion. More precisely,
\begin{align*}
    \sqrt{n}(\check{\theta}_0-\theta_0)
    =
    \sqrt{n}\,\Pn \bar{\psi}_0(\cdot) + O_{\P}(\rho_n)
\end{align*}
uniformly over \(\P \in \cP_n\), where
\[
\Pn := \frac{1}{n}\sum_{i=1}^n \delta_{X_i},
\qquad
\bar{\psi}_0(\cdot) := -J_0^{-1}\psi(\cdot;\theta_0,\eta_0),
\]
and the remainder term satisfies
\begin{align*}
    \rho_n
    &:=
    n^{-1/2+1/q}\log n
    + r_n' \log^{1/2}(1/r_n')
    + n^{1/2}\lambda_n
    + n^{1/2}\lambda_n'
    + n^{1/2}\epsilon_n\\
    &\lesssim \delta_n = o(1).
\end{align*}
As a consequence, uniformly over \(\P \in \cP_n\),
\begin{align*}
    \sup_{t \in \bR^{d_\theta}}
    \Big|
    \P\big(\sqrt{n}(\check{\theta}_0-\theta_0)\le t\big)
    -
    \P\big(N(0,\Sigma^2)\le t\big)
    \Big|
    =
    o(1),
\end{align*}
where the inequality ``\(\le\)'' is understood componentwise, and
\begin{align*}
    \Sigma^2
    :=
    J_0^{-1}
    \E_{\P}\Big[\psi(X;\theta_0,\eta_0)\psi(X;\theta_0,\eta_0)^\top\Big]
    (J_0^{-1})^\top.
\end{align*}
\end{theorem}

\vspace{0.2cm}

We then turn to the bootstrap world. To formulate the bootstrap result, we first introduce the underlying product probability space. Let $$(\Omega_X^\infty,\cA_X^\infty,\P_X^\infty)$$ denote the probability space supporting the data sequence \(\{X_i\}_{i=1}^n\), and let $$(\Omega_W,\cA_W,\P_W)$$ denote the probability space supporting the bootstrap weights. We write \(\P_{XW}\) for the probability measure on the corresponding product space, namely,
\begin{align}\label{eq:ProbSpace}
    \big(\Omega_X^\infty,\cA_X^\infty,\P_X^\infty\big)\times (\Omega_W,\cA_W,\P_W)
    =
    \big(\Omega_X^\infty\times\Omega_W,\cA_X^\infty\times\cA_W,\P_{XW}\big).
\end{align}
Throughout the paper, we assume that the bootstrap weights are independent of the data, so that $$\P_{XW}=\P_X^\infty\times\P_W.$$ Correspondingly, we write \(\E_{XW}^o\) for outer expectation with respect to \(\P_{XW}\), and use analogous notation for \(\E_{W|X}^o\), \(\E_X^o\), and \(\E_W\). For an outer probability measure \(\P^o\), we use the associated stochastic-order notation \(O_{\P}^o(1)\) and \(o_{\P}^o(1)\). For the uniform statements below, one may equivalently fix an arbitrary sequence \(\{\Q_n\}_{n\ge 1}\) with \(\Q_n\in\mathcal P_n\) and interpret the row-wise data law as \(\Q_n^\infty\); we suppress this dependence on \(n\) to simplify notation.

Hereafter, define the bootstrap rate term
\[
a_n
:=
n^{-1/2}\E\Big[\max_{1\le i\le n} W_i\Big],
\]
which is known to converge to zero under Assumption~\ref{asp: BTWeights} \citep[Lemma 4.7]{praestgaard1993exchangeably}.
We also define the bootstrap empirical process operator, acting on any measurable function \(f\), by
\[
\bG_n^*[f]
:=
\frac{1}{\sqrt{n}}\sum_{i=1}^n (W_i-1)f(X_i).
\]

The main result of this paper is the following theorem, which shows that bootstrap consistency holds under exactly the same conditions as those required in Theorem~\ref{thm:DML,Gaussian}. This result can then be leveraged to offer alternative inferential procedures to many methods that are built over the DML framework \citep{lin2023estimation,lin2022regression}.

\begin{theorem}[Bootstrap distribution consistency for DML estimators]
\label{thm:BT,DML,consistency}
Suppose Assumptions~\ref{asp: BT,DML,data}--\ref{asp: BT,DML,2} hold for the original DML estimator, and that the bootstrap weights satisfy Assumption~\ref{asp: BTWeights}. Then the bootstrapped DML estimator satisfies
\begin{align*}
    \sqrt{n}(\check{\theta}_0^*-\check{\theta}_0)
    =
    \bG_n^*[\bar{\psi}_0(\cdot)]
    +
    O_{\P_{XW}}^o(\rho_n^*) \yestag\label{eq:DML,BTGn}
\end{align*}
uniformly over \(\P_X \in \mathcal{P}_n\), where
\begin{align*}
    \rho_n^*
    &:=
    n^{-1/2+1/q}\log n
    +
    r_n' \log^{1/2}(1/r_n')
    +
    n^{1/2}\lambda_n
    +
    n^{1/2}\lambda_n'
    +
    n^{1/2}\epsilon_n
    +
    a_n^{(q-2)/(3q-2)}\log(1/a_n) \\
    & =  o(1).
\end{align*}
Consequently, 
we have, uniformly over \(\P_X \in \cP_n\),
\begin{align*}
    \sup_{t \in \bR^{d_\theta}}
    \Big|
    \P_{W|X}\Big(\sqrt{n}c^{-1}(\check{\theta}_0^*-\check{\theta}_0)\le t\Big)
    -
    \P\big(N(0,\Sigma^2)\le t\big)
    \Big|
    &=
    o_{\P_X}^o(1), \yestag\label{eq:DML,BTconsistency1} \\
    \sup_{t \in \bR^{d_\theta}}
    \Big|
    \P_{W|X}\Big(\sqrt{n}c^{-1}(\check{\theta}_0^*-\check{\theta}_0)\le t\Big)
    -
    \P\Big(\sqrt{n}(\check{\theta}_0-\theta_0)\le t\Big)
    \Big|
    &=
    o_{\P_X}^o(1), \yestag\label{eq:DML,BTconsistency2}
\end{align*}
where \(c^2\) is the limit appearing in Assumption~\ref{asp: BTWeights}(v).
\end{theorem}


In the theorem above, the effect of bootstrap resampling on the distributional approximation is captured by the term \(a_n\). This rate is bootstrap-scheme dependent and is, in general, larger than \(n^{-1/2}\). We close this section with some discussions on this term.

\begin{proposition}[Distributional approximation rates for different bootstrap weights]
\label{Prop:BootstrapRate}
Suppose the bootstrap weights \(W\) satisfy Assumption~\ref{asp: BTWeights}. Then, for any \(x>0\) possibly depending on \(n\), the quantity \(a_n\) satisfies
\begin{align}\label{eq:generate-rate-bootstrap}
    \frac{1}{\sqrt{n}}
    \le
    a_n
    \le
    \frac{x}{\sqrt{n}}
    +
    \frac{\sqrt{n}}{x}
    \Big[\sup_{t \ge x} t^2 \P(W_1 \ge t)\Big].
\end{align}
In particular,
\begin{enumerate}[itemsep=-.5ex, label=(\alph*)]
    \item if \(W_1 \le B_n\) almost surely for some constant \(B_n\) possibly depending on \(n\), then  $a_n \le B_n/\sqrt{n}$;

    \item if \(\P(W_1 \ge t) \lesssim \exp(-bt)\) for some \(b>0\), then  $a_n \lesssim \log n/\sqrt{n}$;

    \item if \(\P(W_1 \ge t) \lesssim t^{-\beta}\) for some \(\beta>2\), then $a_n \lesssim n^{1/\beta - 1/2}$.
\end{enumerate}

In particular, for the bootstrap schemes in Example~\ref{exp: BT, class}, we have:
\begin{enumerate}[itemsep=-.5ex, label=(\roman*)]
    \item for Efron's bootstrap,
    \begin{align*}
        a_n
        =
        (1+o(1))\frac{1}{\sqrt{n}}\frac{\log n}{\log\log n};
    \end{align*}

    \item for normalized multiplier bootstrap with Gamma weights, including the Bayesian bootstrap, as well as for the double bootstrap, we have $a_n \lesssim \log n/\sqrt{n}$;

    \item the optimal rate \(n^{-1/2}\) for \(a_n\) can be attained by some suitably constructed bootstrap weights.
\end{enumerate}
\end{proposition}


\section{Technical details}\label{sec:proof}

\paragraph*{Additional notation.} For presenting the technical details, we first introduce some additional notation. Given any positive integer $n$, we write $n!$ as the factorial of $n$. The notation $\ind(\cdot)$ denotes the indicator function. We use $\stackrel{\sf a.s.}{\longrightarrow}$ and $\stackrel{\sf p}{\longrightarrow}$ to denote almost sure convergence and convergence in probability, respectively. For a particular partition set $I_k$, we use $\BtGk$ to denote the bootstrap empirical measure operator such that 
\[
\BtGk[f] = \sqrt{m}(\BtPk - \Pk)[f] = \frac{1}{\sqrt{m}}\sum_{i \in I_k}(W_{i}-1)f(X_i).
\]

\begin{proof}[Proof of Theorem~\ref{thm:BT,DML,consistency}]

We prove this theorem in six steps. For establishing \eqref{eq:DML,BTGn}, by the definition of $\check{\theta}_0$ and $\check{\theta}_0^*$, we  will show the general arguments for establishing
\begin{align*}
\sqrt{m} (\check{\theta}_{0,k}^* - \check{\theta}_{0,k}) =  - J_0^{-1} \BtGk \big[ \psi(\cdot; {\theta}_{0}, {\eta}_{0}) \big] + O_{\P_{XW}}^o(\rho_n^*) \qquad \text{for all } k\in\zahl{K}
\end{align*}
in Step I-II, and then use Steps III-V to derive the bounds used in Step I-II. We prove the distributional convergence results in \eqref{eq:DML,BTconsistency1} in Step VI; and \eqref{eq:DML,BTconsistency2} follows easily from Theorem~\ref{thm:DML,Gaussian}.
For the most parts of our proof, we implicitly condition on $\{X_i\}_{i \in I_k^c}$ so that $\hat{\eta}_{0,k}$ can be treated as fixed.

To verify the asserted uniformity, fix any sequence $\{\Q_n\}_{n\ge1}$ such that $\Q_n \in \cP_n$ for all $n$. For each $n$, we interpret the generic data-generating law $\P_X$ as $\Q_n$. All stochastic orders and all law-dependent objects (for example, $\theta_0,\eta_0,J_0,\Sigma^2,\bar\psi_0$) are understood as those
associated with $\Q_n$. We suppress this dependence on $n$. Since the sequence $\{\Q_n\}_{n\ge1}$ is arbitrary, the row-wise statements proved below hold uniformly over $\P_X \in \cP_n$ by Lemma~\ref{lemma:UniformSequenceSO}. When subsequence arguments are used in Step VI, we view the resulting triangular array on a common row-wise product space.

\vspace{0.2cm}

{\bf Step I.} First we derive the preliminary rate for the bootstrapped estimator $\check{\theta}_{0,k}^*$. That is, in this step we show that, with $\P_{XW}^o$-probability $1-o(1)$, one has
\begin{align*}
    \big\lVert \check{\theta}_{0,k}^* - {\theta}_0 \big\rVert \le \tau_n. 
    \yestag\label{eq:DML,S1,Goal}
\end{align*}
Observe that, by the approximate $\epsilon_n$-solution condition, we have
\begin{align*}
    \big\lVert \BtPk [\psi(\cdot; \check{\theta}_{0,k}^*, \hat{\eta}_{0,k})] \big\rVert \le \big\lVert \BtPk [\psi(\cdot; \theta_0, \hat{\eta}_{0,k})] \big\rVert + \epsilon_n, ~~~ \epsilon_n = o(\delta_n n^{-1/2}).
\end{align*}
Applying the triangle inequality to the lefthand side, we have 
\begin{align*}
    \big\lVert \P [\psi(\cdot; \check{\theta}_{0,k}^*, {\eta}_{0})] \big\rVert &\le \big\lVert \P [\psi(\cdot; \check{\theta}_{0,k}^*, {\eta}_{0}) - \psi(\cdot; \check{\theta}_{0,k}^*, \hat{\eta}_{0,k})] \big\rVert + \big\lVert (\P-\Pk) [\psi(\cdot; \check{\theta}_{0,k}^*, \hat{\eta}_{0,k})] \big\rVert \\
    &~~+ \big\lVert (\Pk - \BtPk) [\psi(\cdot; \check{\theta}_{0,k}^*, \hat{\eta}_{0,k})] \big\rVert + \big\lVert \BtPk [\psi(\cdot; \check{\theta}_{0,k}^*, \hat{\eta}_{0,k})] \big\rVert
\end{align*}
along with
\begin{align*}
    \big\lVert \BtPk [\psi(\cdot; \theta_0, \hat{\eta}_{0,k})] \big\rVert &\le \big\lVert (\BtPk - \Pk) [\psi(\cdot; {\theta}_{0}, \hat{\eta}_{0,k})] \big\rVert + \big\lVert ( \Pk - \P) [\psi(\cdot; {\theta}_{0}, \hat{\eta}_{0,k})] \big\rVert \\
    & ~~+ \big\lVert \P [\psi(\cdot; {\theta}_{0}, \hat{\eta}_{0,k}) - \psi(\cdot; {\theta}_{0}, {\eta}_{0})] \big\rVert + \big\lVert \P [\psi(\cdot; {\theta}_{0}, {\eta}_{0})] \big\rVert
\end{align*}
on the righthand side. Combine the terms and note the moment condition yields
\begin{align*}
    \big\lVert \P [\psi(\cdot; \check{\theta}_{0,k}^*, {\eta}_{0})] \big\rVert &\le 2T_1 + 2T_2 + 2T_3 + \epsilon_n,
\end{align*}
where
\begin{align*}
    T_1 &:= \sup_{\theta \in \Theta, \eta \in \mathcal{T}_n} \big\lVert \P [\psi(\cdot; {\theta}, {\eta}) - \psi(\cdot; {\theta}, {\eta}_{0})] \big\rVert, \\ 
    T_2 &:= m^{-1/2} \sup_{\theta \in \Theta}  \big\lVert \Gk[\psi(\cdot; {\theta}, \hat{\eta}_{0,k})] \big\rVert, \\
    T_3 &:= m^{-1/2} \sup_{\theta \in \Theta}  \big\lVert \BtGk[\psi(\cdot; {\theta}, \hat{\eta}_{0,k})] \big\rVert.
\end{align*}
Here, by Assumption~\ref{asp: BT,DML,2}(iii), one has 
\[
T_1 = r_n \le \delta_n \tau_n = o(\tau_n). 
\]
Moreover, conditioning on $\{X_i\}_{i \in I_k^c}$, the nuisance estimator $\hat{\eta}_{0,k}$ can be treated as fixed. Lemma~\ref{lemma:MaximalIneq} then yields, with $\P_{X}$-probability $1-o(1)$,
\begin{align*}
    T_2 \lesssim n^{-1/2}(1+n^{-1/2+1/q}\log n) \lesssim n^{-1/2} = o(\tau_n).
\end{align*}
In Step III below, we will show that 
\[
T_3 = o(\tau_n)~~ \text{ with $\P_{XW}^o$-probability goes to $1$.}
\]
Since $\epsilon_n = o(\delta_n n^{-1/2}) = o(\tau_n)$ as $n^{-1/2}\log n \le \tau_n$, by Assumption \ref{asp: BT,DML,1}(iii) we conclude that, with $\P_{XW}^o$-probability $1-o(1)$,
\begin{align*}
    \big\lVert J_0 (\check{\theta}_{0,k}^*- \theta_0) \big\rVert \wedge c_0 \le 2 \big\lVert \P [\psi_{ }(\cdot; \check{\theta}_{0,k}^*, \eta_0)] \big\rVert = o(\tau_n).
\end{align*}
Since the singular values of $J_0$ is bounded away from zero, we then have proved \eqref{eq:DML,S1,Goal} with the same probability. Combined with the proof of Theorem~\ref{thm:DML,Gaussian} in \cite{chernozhukov2018double}, which gives a similar treatment to $\check{\theta}_{0,k}$, one has the following event holds in $\P_{XW}^o$-probability goes to $1$:
\begin{align*}
    \Big\{ 
    \hat{\eta}_{0,k} \in \cT_n,
    ~\big\lVert \check{\theta}_{0,k} - {\theta}_0 \big\rVert \vee \big\lVert \check{\theta}_{0,k}^* - {\theta}_0 \big\rVert \le \tau_n \Big\};
\end{align*}
and we will work on this event in the subsequent steps.

\vspace{0.2cm}

{\bf Step II.} Here we prove the bootstrap consistency by a linearization argument. First note that, for any $\theta \in \Theta$ and $\eta \in \cT_n$, by Taylor's expansion of the function $$r \mapsto \P\big[ \psi (\cdot; \theta_0 + r(\theta-\theta_0), \eta_0 + r(\eta- \eta_0) )\big],$$ one has 
\begin{align*}
    \P\big[ \psi (\cdot; \theta, \eta ) \big] &= J_0(\theta-\theta_0) + \partial_\eta \P \big[\psi(\cdot;\theta_0,\eta_0)\big][\eta- \eta_0] \\
    & ~~+ \int_0^1 (1-r) \partial_r^2 \P\big[ \psi \big(\cdot; \theta_0 + r(\theta-\theta_0), \eta_0 + r(\eta- \eta_0) \big) \big] \d r \yestag \label{eq:DML,S2,lin}
\end{align*}
as $\P[ \psi(\cdot; \theta_0 + r(\theta-\theta_0), \eta_0 + r(\eta- \eta_0) )]$ vanishes at $r=0$. Apply this equality to both $(\check{\theta}_{0,k}, \hat{\eta}_{0,k})$ and $(\check{\theta}_{0,k}^*, \hat{\eta}_{0,k})$, respectively, we note that by Neyman $\lambda_n$-orthogonality, 
\begin{align*}
    \big\lVert \partial_\eta \P[ \psi (\cdot; \theta_0, \eta_0)] [\hat{\eta}_{0,k} - \eta_0] \big\rVert \le \lambda_n.
\end{align*}
At the same time, by simple algebra we can write
\begin{align*}
    & \sqrt{m} \P\big[ \psi(\cdot; \check{\theta}_{0,k}^*, \hat{\eta}_{0,k}) - \psi(\cdot; \check{\theta}_{0,k}, \hat{\eta}_{0,k}) \big] + \BtGk \big[ \psi(\cdot; {\theta}_{0}, {\eta}_{0}) \big] \\
    &= - \BtGk \big[ \psi(\cdot; \check{\theta}_{0,k}^*, \hat{\eta}_{0,k}) - \psi(\cdot; {\theta}_{0}, {\eta}_{0}) \big] + \Gk \big[ \psi(\cdot; \check{\theta}_{0,k}, \hat{\eta}_{0,k}) - \psi(\cdot; {\theta}_{0}, {\eta}_{0}) \big] \\
    & ~~+ \Gk \big[ \psi(\cdot; {\theta}_{0}, {\eta}_{0}) - \psi(\cdot; \check{\theta}_{0,k}^*, \hat{\eta}_{0,k}) \big] 
    + \sqrt{m} \BtPk \big[\psi(\cdot; \check{\theta}_{0,k}^*, \hat{\eta}_{0,k}) \big] 
    - \sqrt{m} \Pk \big[\psi(\cdot; \check{\theta}_{0,k}, \hat{\eta}_{0,k}) \big], \yestag \label{eq:DML,S2,emp}
\end{align*}
and by the approximate $\epsilon_n$-solution condition, for the last two terms in \eqref{eq:DML,S2,emp} one has
\begin{align*}
    \lVert \Pk [\psi(X; \check{\theta}_{0,k}, \hat{\eta}_{0,k})] \rVert & \le \inf_{\theta \in \Theta} \lVert \Pk [\psi(X; \theta, \hat{\eta}_{0,k})] \rVert + \epsilon_n, \\
    \lVert \BtPk [\psi(X; \check{\theta}_{0,k}^*, \hat{\eta}_{0,k})] \rVert & \le \inf_{\theta \in \Theta} \lVert \BtPk [\psi(X; \theta, \hat{\eta}_{0,k})] \rVert + \epsilon_n.
\end{align*}
Combining \eqref{eq:DML,S2,lin} and \eqref{eq:DML,S2,emp} and arranging the terms, we then obtain, with probability $1-o(1)$,
\begin{align*}
    \sqrt{m} \big\lVert J_0\big( \check{\theta}_{0,k}^* - \check{\theta}_{0,k} \big) + \BtGk \big[ \psi(\cdot; {\theta}_{0}, {\eta}_{0}) \big] \big\rVert \le T_4 + T_5 + 2T_6 + T_7 + 2T_8 + 2\sqrt{m}\lambda_n + 2\sqrt{m}\epsilon_n,
\end{align*}
where
\begin{align*}
    T_4 &:= \inf_{\theta \in \Theta} \sqrt{m} \lVert \Pk [\psi(X; \theta, \hat{\eta}_{0,k})] \rVert,\\ 
    T_5 &:= \inf_{\theta \in \Theta} \sqrt{m} \lVert \BtPk [\psi(X; \theta, \hat{\eta}_{0,k})] \rVert, \\
    T_6 &:= \sup_{\theta: \lVert \theta - \theta_0 \rVert \leq \tau_n} \big\lVert \Gk \big[ \psi(\cdot; {\theta}, \hat{\eta}_{0,k}) - \psi(\cdot; {\theta}_{0}, {\eta}_{0}) \big] \big\rVert,  \\
    T_7 &:= \sup_{\theta: \lVert \theta - \theta_0 \rVert \leq \tau_n} \big\lVert \BtGk \big[ \psi(\cdot; {\theta}, \hat{\eta}_{0,k}) - \psi(\cdot; {\theta}_{0}, {\eta}_{0}) \big] \big\rVert,
\end{align*}
and 
\begin{align*}
    T_8 &:= \sup_{\theta: \|\theta - \theta_0\| \le \tau_n} \sup_{\eta\in\mathcal{T}_n} \sqrt{m} \Big\lVert \int_0^1(1-r) \partial_r^2 \P\big[ \psi \big(\cdot; \theta_0 + r(\theta-\theta_0), \eta_0 + r(\eta- \eta_0) \big) \big] \d r \Big\rVert. 
\end{align*}
In the following steps, we will establish bounds on the rates for each of these terms. In particular, in Step IV we will show that, with probability $1-o(1)$,
\begin{align*}
    T_6 \le r_n' \log^{1/2}(1/r_n') + n^{1/q-1/2} \log n \yestag \label{eq:DML,S2,T6eq}
\end{align*}
and
\begin{align*}
    T_7 = O_{\P_{XW}}^o(r_n' \log^{1/2}(1/r_n') + a_n^{(q-2)/(3q-2)}\log(1/a_n)). \yestag\label{eq:DML,S2,T7eq}
\end{align*}
Moreover, by Assumption \ref{asp: BT,DML,2}(iii), we have 
\[
T_8 \le \sqrt{n}\lambda_n', 
\]
and finally, we will show
\begin{align*}
    T_4 &\le \sqrt{n} \lambda_n + T_8 + T_6, \yestag  \label{eq:DML,S2,T4bd}\\
    T_5 &\le \sqrt{n} \lambda_n + T_8 + T_7 + T_6, \yestag  \label{eq:DML,S2,T5bd}
\end{align*}
with probability $1-o(1)$ in Step V. Combining \eqref{eq:DML,S2,T6eq}-\eqref{eq:DML,S2,T5bd}, we then obtain
\begin{align*}
    \sqrt{m} (\check{\theta}_{0,k}^* - \check{\theta}_{0,k}) =  - J_0^{-1} \BtGk \big[ \psi(\cdot; {\theta}_{0}, {\eta}_{0}) \big] + O_{\P_{XW}}^o(\rho_n^*),
\end{align*}
where
\begin{align*}
    \rho_n^* := n^{-1/2+1/q}\log n + r'_n \log^{1/2} (1/r'_n)  + n^{1/2} (\lambda_n + \lambda'_n + \epsilon_n) + a_n^{(q-2)/(3q-2)}\log(1/a_n) = o(1).
\end{align*}

\vspace{0.2cm}

{\bf Step III.} Here we derive the rate for $T_3$. To prove 
\[
T_3 = o(\tau_n) ~~\text{ with $\P_{XW}^o$-probability $1-o(1)$,}
\]
it suffices to show 
\[
T_3 = o_{\P_{XW}}^o(\tau_n). 
\]
We illustrate this by writing $T_3$  as the summation of a term of order $o(\tau_n)$ with high probability and a term of order $O_{\P_{XW}}(n^{-1/2})$. The order of $\tau_n$ then yields the claimed result.

Notice that by Assumption \ref{asp: BTWeights}(ii), for any measurable $f$, we can write
\begin{align*}
    \BtGk[f] &= \frac{1}{\sqrt{m}} \sum_{i\in I_k} (W_i-1)\delta_{X_i}[f] \\
    &= \frac{1}{\sqrt{m}} \sum_{i\in I_k} (W_i-1)(\delta_{X_i} - \P_X)[f] + \frac{1}{\sqrt{m}} \sum_{i\in I_k} (W_i-1) \P_X[f] \\
    &=: \BtGko[f] + \Delta_{n,k}\P_X[f];
\end{align*}
and correspondingly, we can write 
\begin{align*}
    T_3 \le m^{-1/2} \sup_{\theta \in \Theta}  \big\lVert \BtGko[\psi(\cdot; {\theta}, \hat{\eta}_{0,k})] \big\rVert + m^{-1/2} \sup_{\theta \in \Theta}  \big\lVert \Delta_{n,k}\P_X[\psi(\cdot; {\theta}, \hat{\eta}_{0,k})] \big\rVert =: T_{3,c} + T_{3,r} .\yestag \label{eq:DML,S3,T3decom}
\end{align*}
For the first term in \eqref{eq:DML,S3,T3decom}, we show 
\[
T_{3,c} = o(\tau_n)~~\text{ with $\P_{XW}^o$-probability $1-o(1)$.} 
\]
Since by Markov inequality 
\begin{align*}
    \P_{XW}^o \big(T_{3,c} \le n^{-1/2} \log \log (n) \big) \ge 1 - \frac{\sqrt{n} \E_{XW}^o[T_{3,c}]}{ \log \log (n)},
\end{align*}
 and by definition of $\tau_n$ one has $T_{3,c} / \tau_n \le \log \log (n)/\log n = o(1)$ on the same event defined in the probability, it only suffices to show
\begin{align*}
    \limsup_{n \to \infty} \sqrt{m} \E_{XW}^o [ T_{3,c} ] < \infty. \yestag \label{eq:DML,S3,Goal}
\end{align*}
Given an independent bootstrap weights copy $\bar{W}=(\bar W_i)_{i=1}^n$, since by Assumption \ref{asp: BTWeights}(i)(ii) one has $\E_{\bar{W}}[\bar{W}_i]=1$ for all $i \in I_k$,
we can then leverage the symmetrization argument to establish
\begin{align*}
    \sqrt{m} \E_{XW}^o [ T_{3,c} ] &= \E_{XW}^o \Big[  \sup_{\theta \in \Theta}  \big\lVert \BtGko[\psi(\cdot; {\theta}, \hat{\eta}_{0,k})] \big\rVert \Big] \\
    &= \E_{XW}^o \Big[ \sup_{\theta \in \Theta}  \Big\lVert \frac{1}{\sqrt{m}}\sum_{i\in I_k} (W_i-1)(\delta_{X_i} - \P_X)[\psi(\cdot; {\theta}, \hat{\eta}_{0,k})] \Big\rVert \Big] \\
    &= \E_{XW}^o \Big[ \sup_{\theta \in \Theta}  \Big\lVert \frac{1}{\sqrt{m}}\sum_{i\in I_k} \big(W_i-\E_{\bar{W}}[\bar{W}_i]\big)(\delta_{X_i} - \P_X)[\psi(\cdot; {\theta}, \hat{\eta}_{0,k})] \Big\rVert \Big] \\
    & \le \E_{XW}^o \E_{\bar{W}}^o \Big[ \sup_{\theta \in \Theta}  \Big\lVert \frac{1}{\sqrt{m}}\sum_{i\in I_k} \big(W_i-\bar{W}_i'\big)(\delta_{X_i} - \P_X)[\psi(\cdot; {\theta}, \hat{\eta}_{0,k})] \Big\rVert \Big] \\
    & \le \E_{XW}^o \Big[ \sup_{\theta \in \Theta}  \Big\lVert \frac{1}{\sqrt{m}}\sum_{i\in I_k} W_i(\delta_{X_i} - \P_X)[\psi(\cdot; {\theta}, \hat{\eta}_{0,k})] \Big\rVert \Big] \\
    & ~~+ \E_{XW}^o \E_{\bar{W}}^o\Big[ \sup_{\theta \in \Theta}  \Big\lVert \frac{1}{\sqrt{m}}\sum_{i\in I_k} \bar{W}_i(\delta_{X_i} - \P_X)[\psi(\cdot; {\theta}, \hat{\eta}_{0,k})] \Big\rVert \Big] \\
    & = 2\E_{XW}^o \Big[ \sup_{\theta \in \Theta}  \Big\lVert \frac{1}{\sqrt{m}}\sum_{i\in I_k} W_i(\delta_{X_i} - \P_X)[\psi(\cdot; {\theta}, \hat{\eta}_{0,k})] \Big\rVert \Big] \\
    &= 2\E_{XW}^o \Big[ \sup_{\theta \in \Theta}  \Big\lVert \frac{1}{\sqrt{m}}\sum_{i\in I_k} W_i Z_i(\theta) \Big\rVert \Big],
\end{align*}
where we define 
\begin{align*}
    Z_i(\theta) := (\delta_{X_i} - \P_X)[\psi(\cdot; {\theta}, \hat{\eta}_{0,k})]
\end{align*}
for $i \in I_k$. Re-indexing the samples in the fold, by Lemma \ref{lemma:MultIneq} we have, for any fixed $m_0 < \infty$,  
\begin{align*}
    \limsup_{n \to \infty} \E_{XW}^o \Big[ \sup_{\theta \in \Theta}  \Big\lVert \frac{1}{\sqrt{m}}\sum_{i=1}^{m} W_i Z_i(\theta) \Big\rVert \Big] &\le 
    m_0 \limsup_{n \to \infty}  \E_{X}^o \Big[ \sup_{\theta \in \Theta}  \big\lVert Z_1(\theta) \big\rVert \Big] \Big(\frac{\E_{W}\big[ \max_{1\le i \le m} W_i \big]}{\sqrt{m}}\Big) \\
    & ~~+  \limsup_{n \to \infty} \Big\{ S_n  \max_{m_0 \le i \le m} \E_{X}^o \Big[ \sup_{\theta \in \Theta} \Big\lVert \frac{1}{\sqrt{i}} \sum_{j=m_0+1}^i Z_j(\theta) \Big\rVert \Big] \Big\},
\end{align*}
where
\begin{align*}
    S_n := \int_0^\infty \sqrt{\P_W(W_1\ge t)}\d t. 
    \yestag \label{eq:DML,S3,SnDef}
\end{align*}
To show \eqref{eq:DML,S3,Goal}, it then only suffices to show
\begin{align*}
    &\limsup_{n \to \infty} \E_{X}^o \Big[ \sup_{\theta \in \Theta}  \big\lVert Z_1(\theta) \big\rVert \Big] < \infty, \yestag \label{eq:DML,S3,sG1} \\
    &\limsup_{n \to \infty} a_n = \limsup_{n \to \infty}\Big(\frac{\E_{W}\big[ \max_{1\le i \le n} W_i \big]}{\sqrt{n}}\Big) = 0, \yestag \label{eq:DML,S3,sG2}\\
    &\limsup_{n \to \infty} S_n < \infty, \yestag \label{eq:DML,S3,sG3}\\
    &\limsup_{ n \to \infty} \max_{m_0 < i \le m}  \E_{X} \Big[ \sup_{\theta \in \Theta} \Big\lVert \frac{1}{\sqrt{i}} \sum_{j=m_0+1}^i Z_j(\theta) \Big\rVert \Big] < \infty \yestag \label{eq:DML,S3,sG4}.
\end{align*}
To establish \eqref{eq:DML,S3,sG1}, by Assumption \ref{asp: BT,DML,2}(ii), we can write
\begin{align*}
    \limsup_{n \to \infty} \E_{X}^o \Big[ \sup_{\theta \in \Theta}  \big\lVert Z_1(\theta) \big\rVert \Big] &= \limsup_{n \to \infty} \E_{X}^o \Big[ \sup_{\theta \in \Theta}  \big\lVert (\delta_{X_1} - \P_X)[\psi(\cdot; {\theta}, \hat{\eta}_{0,k})] \big\rVert \Big] \\
    & \le \limsup_{n \to \infty} \Big\{ \E_{X}^o \Big[ \sup_{\theta \in \Theta}  \big\lVert \psi(X_1; {\theta}, \hat{\eta}_{0,k}) \big\rVert \Big] + \sup_{\theta \in \Theta}  \big\lVert \P_X [\psi(\cdot; {\theta}, \hat{\eta}_{0,k})] \big\rVert  \Big\} \\
    & \lesssim \limsup_{n \to \infty} 2 \E_{X} \big[F_{1,\hat{\eta}_{0,k}}(X_1)\big]\\
    & \le 2c_1.
\end{align*}

Equation \eqref{eq:DML,S3,sG2} is a standard implication for exchangeable bootstrap weights under Assumptions \ref{asp: BTWeights}(iii)(iv). To see this, either refer to our proof in Step~VI or directly apply Lemma 4.7 in \cite{praestgaard1993exchangeably}. 

Equation~\eqref{eq:DML,S3,sG3} is a direct result of $S_n$ is uniformly bounded for all $n$ by Assumption~\ref{asp: BTWeights}(iii). 

To show \eqref{eq:DML,S3,sG4}, we further re-index 
\[
s:=i-m_0 \text{ for } i>m_0, 
\]
and therefore
\begin{align*}
     \max_{m_0 < i \le m} \E_{X} \Big[ \sup_{\theta \in \Theta} \Big\lVert \frac{1}{\sqrt{i}} \sum_{j=m_0+1}^i Z_j(\theta) \Big\rVert \Big]  
     &=  \max_{0 < s \le m-m_0}  \sqrt{\frac{s}{s+m_0}} \E_{X} \Big[ \sup_{\theta \in \Theta} \Big\lVert \frac{1}{\sqrt{s}} \sum_{j=1}^s Z_j(\theta) \Big\rVert \Big] \\
     & \lesssim  \max_{0 < s \le m-m_0}  \E_{X} \Big[ \lVert \mathbb{G}_s \rVert_{\mathcal{F}_{1,\hat \eta_{0,k}}} \Big],
\end{align*}
where for each $s$, the re-indexed empirical process operator is defined as 
\begin{align*}
    \mathbb{G}_s := \frac{1}{\sqrt{s}} \sum_{j=1}^s (\delta_{X_j} - \P_X). 
    \yestag\label{eq:DML,S3,GsDef}
\end{align*}
Applying Lemma~\ref{lemma:MaximalIneq} with sample size $s$ and noting Assumption~\ref{asp: BT,DML,2}(ii), we then obtain
\begin{align*}
    \max_{m_0 < i \le m} \E_{X} \Big[ \sup_{\theta \in \Theta} \Big\lVert \frac{1}{\sqrt{i}} \sum_{j=m_0+1}^i Z_j(\theta) \Big\rVert \Big] 
    \lesssim  \max_{0 < s \le m-m_0} \big( 1+s^{-1/2+1/q} c_1 \big) \lesssim 1,
\end{align*}
which concludes \eqref{eq:DML,S3,sG4}.

For the second term in \eqref{eq:DML,S3,T3decom}, we aim to show 
\[
T_{3,r} = O_{\P_{XW}}(n^{-1/2}). 
\]
To this end, note that 
\begin{align*}
    T_{3,r} = \sup_{\theta \in \Theta}  \Big\lVert \frac{1}{m} \sum_{i\in I_k} (W_i-1) \P_X[\psi(\cdot; {\theta}, \hat{\eta}_{0,k})]  \Big\rVert \lesssim \Big\lvert \frac{1}{m} \sum_{i\in I_k} (W_i-1)  \Big\rvert \lVert F_{1,\hat \eta_{0,k}}\rVert_{\P,q}.
\end{align*}
Then, from Assumption~\ref{asp: BT,DML,2}(ii), we only have to show $$\frac{1}{m}\sum_{i\in I_k} (W_i-1) = O_{\P}(n^{-1/2}).$$ Observe that the fact
\begin{align*}
    0 = \Var\Big[\sum_{i=1}^n W_i\Big] =  n\Var[ W_i] + n(n-1) \Cov[W_1,W_2]
\end{align*}
yields
\begin{align*}
     \E\Big[ \Big(\frac{1}{\sqrt{m} }\Delta_{n,k} \Big)^2 \Big] = \Var\Big[\frac{1}{m} \sum_{i\in I_k} (W_i-1)\Big] = \frac{n-m}{m(n-1)} \Var[W_1] = O(n^{-1}).
\end{align*}
To see the last equality, we prove that (iii)(iv) of Assumption \ref{asp: BTWeights} implies the sequence $\{W_i\}_{i=1}^n$ is uniformly square-integrable: for any $\lambda$ sufficiently large such that $u^2 \P(W_{i} \ge u) \le \epsilon^2$ for large $n$ when $u \ge \lambda$, write
\begin{align*}
    \E\big[W_{i}^2 \ind(W_{i} \ge \lambda)\big] &= 2 \int_0^\infty u \P(W_{i} \ge u \vee \lambda) \d u \\
    &= \lambda^2 \P(W_{i} \ge \lambda) + 2 \int_{\lambda}^\infty u \P(W_{i} \ge u) \d u \\
    &\le \epsilon^2 + 2 \Big[ \sup_{u \ge \lambda}  u \sqrt{\P(W_{i} \ge u)}  \Big]\int_{\lambda}^\infty \sqrt{\P(W_{i} \ge u)} \d u \\
    &\le \epsilon^2 + 2C \epsilon.
\end{align*}
The arbitrariness of $\epsilon$ then yields uniform square-integrablility, that is,
\begin{align*}
    \lim_{\lambda \to \infty} \limsup_{n \to \infty} \E\big[W_{i}^2 \ind(W_{i} \ge \lambda)\big] = 0, \yestag \label{eq:DML,BTweights,USI}
\end{align*}
which implies the existence of $W_i$'s second moment and concludes $\Delta_{n,k} = O_{\P}(1)$. Together we have established the rate for $T_3$ as claimed.

\vspace{0.2cm}

{\bf Step IV.} Here we first derive the bound \eqref{eq:DML,S2,T6eq} for $T_6$. For any fixed $\eta \in \mathcal{T}_n$, define 
\begin{align*}
    \mathcal{F}_{2,\eta} := \big\{ \psi_j(\cdot; \theta, \eta) - \psi_j(\cdot; \theta_0, \eta_0): j \in \zahl{d_{\theta}}, \|\theta - \theta_0\| \le \tau_n \big\}.
\end{align*}
It is then immediate that 
\[
T_6 \lesssim \lVert \Gk \rVert_{\mathcal{F}_{2,\hat\eta_{0,k}}}. 
\]
Obviously, $F_{2,\eta} := F_{1,\eta} + F_{1,\eta_0}$ is an envelope function for $\mathcal{F}_{2,\eta}$ and satisfies the moment condition $$\lVert F_{2,\eta} \rVert_{\P,q} \le \lVert F_{1,\eta} \rVert_{\P,q} + \lVert F_{1,\eta_0} \rVert_{\P,q} = 2c_1$$ by the triangle inequality. We can also find a constant $C'$ large enough so that $$\sup_{f \in \mathcal{F}_{2,\eta}} \| f \|_{\P,2} \le C' r_n' \le \| F_{2,\eta} \|_{\P,2}$$ by Assumption~\ref{asp: BT,DML,2}(i)(iii). Moreover, since $\mathcal{F}_{2,\eta} \subseteq \mathcal{F}_{1,\eta} - \mathcal{F}_{1,\eta_0}$, the proof of Theorem 3 in \cite{andrews1994empirical} yields
\begin{align*}
    & \sup_\Q  \log N \big(\epsilon \|F_{2,\eta}\|_{\Q,2}, \mathcal{F}_{2,\eta}, \| \cdot \|_{\Q,2} \big) \le \sup_\Q  \log N \big(\epsilon \|F_{2,\eta}\|_{\Q,2}, \mathcal{F}_{1,\eta} - \mathcal{F}_{1,\eta_0}, \| \cdot \|_{\Q,2} \big) \\
    & ~ \le \sup_\Q  \log N \big((\epsilon/2) \|F_{1,\eta}\|_{\Q,2}, \mathcal{F}_{1,\eta}, \| \cdot \|_{\Q,2} \big) + \sup_\Q  \log N \big((\epsilon/2) \|F_{1,\eta_0}\|_{\Q,2}, \mathcal{F}_{1,\eta_0}, \| \cdot \|_{\Q,2} \big),
\end{align*}
which is upper bounded by $2v \log(2a/\epsilon)$ for all $0<\epsilon \le 1$ by Assumption~\ref{asp: BT,DML,2}(ii). Therefore, applying Lemma \ref{lemma:MaximalIneq} conditioning on $\{X_i\}_{i \in I_k^c}$, we have, with $\P_{X}^o$-probability $1-o(1)$,
\begin{align*}
     \sup_{f \in \mathcal{F}_{2,\eta}} \lvert \Gk(f) \rvert \lesssim r_n' \log^{1/2}(1/r_n') + n^{1/q-1/2} \log n.
\end{align*}
Replacing the $\eta$ with $\hat \eta_{0,k}$ concludes \eqref{eq:DML,S2,T6eq}. 

Next, we prove the probability bound \eqref{eq:DML,S2,T7eq} for $T_7$. Similarly as in Step III, we can write 
\begin{align*}
   T_7 &\le \sup_{\theta: \lVert \theta - \theta_0 \rVert \leq \tau_n} \big\lVert \BtGko \big[ \psi(\cdot; {\theta}, \hat{\eta}_{0,k}) - \psi(\cdot; {\theta}_{0}, {\eta}_{0}) \big]  \big\rVert + \sup_{\theta: \lVert \theta - \theta_0 \rVert \leq \tau_n} \big\lVert \P\big[ \psi(\cdot; {\theta}, \hat{\eta}_{0,k}) - \psi(\cdot; {\theta}_{0}, {\eta}_{0}) \big]  \big\rVert \lvert \Delta_{n,k} \rvert \\
   & =: T_{7,c} + T_{7,r}.
\end{align*}
For the second term above, note that Assumption~\ref{asp: BT,DML,2}(iii) and $\Delta_{n,k} = O_{\P}(1)$, which was shown in Step III, immediately gives 
\[
T_{7,r} = O_{\P_{XW}}(r_n').
\]
It then suffices to show
\begin{align*}
    \E_{XW}^o [ T_{7,c} ] \lesssim r_n' \log^{1/2}(1/r_n') + a_n^{(q-2)/(3q-2)}\log(1/a_n). \yestag \label{eq:DML,S4,SG2}
\end{align*}
With a similar symmetrization argument as in step III, for any fixed $\eta \in \cT_n$ we have
\begin{align*}
     \E_{XW}^o \Big[  \sup_{\theta: \lVert \theta - \theta_0 \rVert \leq \tau_n} \big\lVert \BtGko[\psi(\cdot; {\theta}, \eta) - \psi(\cdot; {\theta}_0, {\eta}_{0})] \big\rVert \Big] \le 2\E_{XW}^o \Big[ \sup_{\theta: \lVert \theta - \theta_0 \rVert \leq \tau_n}  \Big\lVert \frac{1}{\sqrt{m}}\sum_{i=1}^{m} W_i Z_i'(\theta) \Big\rVert \Big],  \yestag \label{eq:DML,S4,symmetrization}
\end{align*}
where we define 
\begin{align*}
    Z_i'(\theta) := \big( \delta_{X_i} - \P_X \big) \big[\psi(\cdot; {\theta}, \eta) - \psi(\cdot; {\theta}_0, {\eta}_{0}) \big].
\end{align*}
Again, by the multiplier inequality, for any $m_0 < \infty$, we have 
\begin{align*}
     \E_{XW}^o \Big[ \sup_{\theta: \lVert \theta - \theta_0 \rVert \leq \tau_n}  \Big\lVert \frac{1}{\sqrt{m}}\sum_{i=1}^{m} W_i Z_i'(\theta) \Big\rVert \Big] &\le m_0 \E_{X}^o \Big[ \sup_{\theta: \lVert \theta - \theta_0 \rVert \leq \tau_n}  \big\lVert Z_1'(\theta) \big\rVert \Big] \Big(\frac{\E_{W}\big[ \max_{1\le i \le m} W_i \big]}{\sqrt{m}}\Big) \\
    & ~~+  S_n  \max_{m_0 \le i \le m} \E_{X}^o \Big[ \sup_{\theta: \lVert \theta - \theta_0 \rVert \leq \tau_n} \Big\lVert \frac{1}{\sqrt{i}} \sum_{j=m_0+1}^i Z_j'(\theta) \Big\rVert \Big] \yestag \label{eq:DML,S4,MultIneq}
\end{align*}
with $S_n$ identically defined as in \eqref{eq:DML,S3,SnDef}. Similarly as \eqref{eq:DML,S3,sG2}, we have, uniformly for all $n$,
\begin{align*}
     \E_{X}^o \Big[ \sup_{\theta: \lVert \theta - \theta_0 \rVert \leq \tau_n}  \big\lVert Z_1'(\theta) \big\rVert \Big] \le 
     2 \sqrt{d_\theta} \E_{X}[F_{2,\eta}] \leq 4c_1\sqrt{d_\theta}.
\end{align*}
Recall $a_n = n^{-1/2}\E[\max_{1\le i \le n}W_i]$ converge to $0$ as $n$ goes to infinity, and by Assumption~\ref{asp: BTWeights}(iii) $S_n$ is uniformly bounded. Re-indexing $s:=i-m_0$ and invoking the definition of $\mathbb{G}_s$ in \eqref{eq:DML,S3,GsDef}, we bound the last term in \eqref{eq:DML,S4,MultIneq} by applying Lemma~\ref{lemma:MaximalIneq} and a similar argument as we did for $T_6$:
\begin{align*}
    & \max_{1 \le s \le m-m_0} \Big\{ \sqrt{\frac{s}{s+m_0}}\E_{X}^o \Big[  \sup_{\theta: \lVert \theta - \theta_0 \rVert \leq \tau_n} \Big\lVert \frac{1}{\sqrt{s}} \sum_{j=1}^s Z_j'(\theta) \Big\rVert \Big] \Big\} 
    \lesssim \max_{1 \le s \le m-m_0} \Big\{ \sqrt{\frac{s}{s+m_0}}\E_{X}^o \Big[ \lVert \mathbb{G}_s \rVert_{\cF_{2,\eta}} \Big] \Big\} \\
    & ~~~ \lesssim \max_{1 \le s \le m-m_0} \sqrt{\frac{s}{s+m_0}} \big\{ r_n' \log^{1/2}(1/r_n') +s^{1/q-1/2} \log(s) \big\} \\
    & ~~~ \le  
    r_n' \log^{1/2}(1/r_n') + \max_{1 \le s \le m-m_0} \sqrt{\frac{s}{s+m_0}} \big\{ s^{1/q-1/2} \log(s)\big\} \\
    & ~~~ \lesssim r_n' \log^{1/2}(1/r_n') + m_0^{1/q-1/2}\log(m_0). \yestag \label{eq:DML,S4,order2}
\end{align*}
 Combining \eqref{eq:DML,S4,symmetrization}-\eqref{eq:DML,S4,order2}, one has, for any $\eta \in \cT_n$,
\begin{align*}
    \E_{XW}^o \Big[  \sup_{\theta: \lVert \theta - \theta_0 \rVert \leq \tau_n} \big\lVert \BtGko[\psi(\cdot; {\theta}, \eta) - \psi(\cdot; {\theta}_0, {\eta}_{0})] \big\rVert \Big] \lesssim m_0 a_n + r_n' \log^{1/2}(1/r_n') + m_0^{1/q-1/2}\log(m_0).
\end{align*}
Setting $m_0(n) := \floor{a_n^{-1/(3/2-1/q)}}{}$ and letting $\eta = \hat \eta_{0,k}$ establishes \eqref{eq:DML,S4,SG2}.

\vspace{0.2cm}

{\bf Step V.} First, we derive the bound \eqref{eq:DML,S2,T4bd} for $T_4$. Let $$\bar{\theta}_{0,k} := \theta_0 - J_0^{-1}\Pk[\psi(\cdot;\theta_0,\eta_0)].$$
Since the singular values of $J_0$ are bounded away from zero and, by
Assumption~\ref{asp: BT,DML,2}(ii) with $\eta_0 \in \cT_n$ we have $\E[ \lVert \Gk[\psi(\cdot;\theta_0,\eta_0)]\rVert^2]\lesssim 1$, Markov's inequality yields
\begin{align*}
    \lVert \bar{\theta}_{0,k} - \theta_0 \rVert \le m^{-1/2} \lVert J_0^{-1} \rVert  \lVert \Gk[\psi(\cdot;\theta_0,\eta_0)] \rVert = O_{\P_X}^o(n^{-1/2}) = o_{\P_X}^o(n^{-1/2} \log n) = o_{\P_X}^o(\tau_n).
\end{align*}
Therefore, by definition of $\Theta$ in Assumption~\ref{asp: BT,DML,1}(i), one has 
\[
\bar{\theta}_{0,k} \in \Theta~~\text{ with $\P_{X}^o$-probability $1-o(1)$.} 
\]
Then with the same probability,
\begin{align*}
    T_4 = \inf_{\theta \in \Theta} \sqrt{m} \big\lVert \Pk [\psi(X; \theta, \hat{\eta}_{0,k})] \big\rVert \le \sqrt{m} \big\lVert \Pk [\psi(X; \bar{\theta}_{0,k}, \hat{\eta}_{0,k})] \big\rVert.
\end{align*}
The righthand side of the above inequality has the upper bound
\begin{align*}
    \sqrt{m} \big\lVert \Pk [\psi(X; \bar{\theta}_{0,k}, \hat{\eta}_{0,k})] \big\rVert &\le \sqrt{m} \big\lVert \Pk [\psi(X; {\theta}_{0}, {\eta}_{0})] + \P [\psi(\cdot; \bar{\theta}_{0,k}, \hat{\eta}_{0,k})]
    \big\rVert \\
    &~~ + \big\lVert \Gk [ \psi(\cdot; \bar{\theta}_{0,k}, \hat{\eta}_{0,k}) - \psi(\cdot; {\theta}_0, {\eta}_{0}) ] \big\rVert \\
    & \le \sqrt{n} \lambda_n + T_8 + T_6,
\end{align*}
where the second inequality is given by \eqref{eq:DML,S2,lin}, the linearization of $\P [\psi(\cdot; \bar{\theta}_{0,k}, \hat{\eta}_{0,k})]$. 

Next, we derive the bound \eqref{eq:DML,S2,T5bd} for $T_5$. Let 
\[
\bar{\theta}_{0,k}^* := \theta_0 - J_0^{-1} \BtPk[\psi(\cdot;\theta_0,\eta_0)]. 
\]
Notice that 
\begin{align*}
    \sqrt{m} \BtPk[\psi(\cdot;\theta_0,\eta_0)] = \BtGk [\psi(\cdot;\theta_0,\eta_0)] + \Gk [\psi(\cdot;\theta_0,\eta_0)] = O_{\P_{XW}}(1),
\end{align*}
as $\Gk [\psi(\cdot;\theta_0,\eta_0)] = O_{\P}(1)$ and we claim $\BtGk [\psi(\cdot;\theta_0,\eta_0)] = O_{\P_{XW}}(1)$: 
for any $j \in \zahl{d_{\theta}}$, the conditional second moment of the $j$-th component of $\BtGk [\psi(\cdot;\theta_0,\eta_0)]$ can be written as 
\begin{align*}
    & \E_{W|X}\Big[\Big(\frac{1}{\sqrt{m}} \sum_{i \in I_k} (W_i-1) \psi_{j}(X_i;\theta_0, \eta_0)\Big)^2\Big] \\
    & = \frac{1}{m} \Big\{ \E\big[(W_1-1)^2\big] \sum_{i \in I_k} \psi_{j}^2(X_i;\theta_0, \eta_0) + \E[(W_1-1)(W_2-1)] \sum_{i, \ell \in I_k, i \neq\ell} \psi_{j}(X_i;\theta_0, \eta_0) \psi_{j}(X_\ell;\theta_0, \eta_0) \Big\} \\
    & = \frac{1}{m} \E\big[(W_1-1)^2\big] \Big\{  \sum_{i \in I_k} \psi_{j}^2(X_i;\theta_0, \eta_0) - \frac{1}{n-1} \sum_{i, \ell \in I_k, i \neq\ell} \psi_{j}(X_i;\theta_0, \eta_0) \psi_{j}(X_\ell;\theta_0, \eta_0) \Big\} \\
    & \le \frac{1}{m} \Big(1+\frac{m-1}{n-1} \Big) \E\big[(W_1-1)^2\big]  \Big( \sum_{i \in I_k} \psi_{j}^2(X_i;\theta_0, \eta_0) \Big)\\
    & \le 2 \E\big[(W_1-1)^2\big] \Pk [\psi_{j}^2(\cdot;\theta_0, \eta_0)],
\end{align*}
where the second equality is a result of expanding $\Var[\sum_{i=1}^n(W_i-1)]=0$ and the first inequality uses the fact that for any $I_k$-indexed sequence of real numbers $\{b_i\}_{i \in I_k}$, one has the basic inequality $\lvert \sum_{i \neq \ell} b_i b_\ell \rvert \le (m-1) \sum_{i \in I_k}b_i^2$. Since $\E[(W_1-1)^2]$ is uniformly bounded by \eqref{eq:DML,BTweights,USI}, applying the law of total expectation and noting Assumption~\ref{asp: BT,DML,2}(ii) completes the proof for the claim. 

We then have
\begin{align*}
    \lVert \bar{\theta}_{0,k}^* - \theta_0 \rVert \le m^{-1/2} \lVert J_0^{-1} \rVert  \lVert m^{1/2} \BtPk[\psi(\cdot;\theta_0,\eta_0)] \rVert = O_{\P_{XW}}(n^{-1/2}) = o_{\P_{XW}}(n^{-1/2} \log n) = o_{\P_{XW}}(\tau_n),
\end{align*}
and we have $\bar{\theta}_{0,k}^* \in \Theta$ with $\P_{XW}^o$-probability $1-o(1)$. With the same probability,
\begin{align*}
    T_5 = \inf_{\theta \in \Theta} \sqrt{m} \lVert \BtPk [\psi(X; \theta, \hat{\eta}_{0,k})] \rVert \le \sqrt{m} \big\lVert \BtPk [\psi(X; \bar{\theta}_{0,k}^*, \hat{\eta}_{0,k})] \big\rVert.
\end{align*}
By the triangle inequality, the righthand side is upper bounded by
\begin{align*}
    \sqrt{m} \big\lVert \BtPk [\psi(X; \bar{\theta}_{0,k}^*, \hat{\eta}_{0,k})] \big\rVert &\le \sqrt{m} \big\lVert \BtPk [\psi(X; {\theta}_{0}, {\eta}_{0})] + \P [\psi(\cdot; \bar{\theta}_{0,k}^*, \hat{\eta}_{0,k})]
    \big\rVert \\
    &~~ + \big\lVert \BtGk [ \psi(\cdot; \bar{\theta}_{0,k}^*, \hat{\eta}_{0,k}) - \psi(\cdot; {\theta}_0, {\eta}_{0}) ] \big\rVert \\
    &~~ + \big\lVert \Gk [ \psi(\cdot; \bar{\theta}_{0,k}^*, \hat{\eta}_{0,k}) - \psi(\cdot; {\theta}_0, {\eta}_{0}) ] \big\rVert \\
    & \le \sqrt{n} \lambda_n + T_8 + T_7 + T_6,
\end{align*}
as we can linearize $\P [\psi(\cdot; \bar{\theta}_{0,k}^*, \hat{\eta}_{0,k})]$ in the same way. 

\vspace{0.2cm}

{\bf Step VI.} In this part we will show that \eqref{eq:DML,BTGn} indeed implies \eqref{eq:DML,BTconsistency1}. Employing Cramér–Wold device, it is equivalent to showing that, for any $z \in \bR^{d_\theta}$, 
\begin{align*}
    \sup_{t \in \bR} \Big\lvert \P_{W|X}\big(c^{-1} \bG_{n}^* [z^\top \bar \psi_0 (\cdot)] \le t \big) - \P\big(N(0, z^\top \Sigma^2 z) \le t\big) \Big\rvert &=  o_{\P_X}^o(1).
\end{align*}
To this end, we adapt the proof of \cite{mason1992rank} to our setup. Define 
\begin{align*}
    H_n := \Big[\frac{1}{n}\sum_{i=1}^n(W_i-1)^2 \frac{1}{n}\sum_{i=1}^n \big(z^\top \bar \psi_0(X_i) - \bP_n [ z^\top \bar{\psi}_0(\cdot)] \big)^{2} \Big]^{-1/2} \bG_{n}^* [z^\top \bar \psi_0 (\cdot)].
\end{align*}
Since one has, by Assumption~\ref{asp: BTWeights}(v) and the weak law of large numbers ,
\begin{align*}
\frac{1}{n} \sum_{i=1}^n (W_i-1)^2 \stackrel{\sf p}{\longrightarrow} c^2 ~\text{ and }~ \frac{1}{n} \sum_{i=1}^n \big(z^\top \bar \psi_0(X_i) - \bP_n [ z^\top \bar{\psi}_0(\cdot)] \big)^{2} = z^\top \Sigma^2 z + o_{\P}(1). \yestag \label{eq:matching,Pcvrg_1}
\end{align*}
To prove \eqref{eq:DML,BTconsistency1}, it then suffices to show
\begin{align*}
    \sup_{t \in \bR} \Big\lvert \P_{W|X}\big(H_n \le t \big) - \P\big(N(0,1) \le t\big) \Big\rvert = o_{\P_X}(1). \yestag \label{eq:matching,ANormalZ_n}
\end{align*}
Let $R=(R_i)_{i=1}^n$ be a random permutation of $\zahl{n}$, uniformly distributed over all $n!$ permutations and independent of $(X,W)$. Writing $\{W_{R_i}\}_{i=1}^n$ to be the permuted weights of $\{W_i\}_{i=1}^n$ according to $R$, we denote the permuted version of $H_n$ as
\begin{align*}
    H_n^R &:= \Big[\sum_{i=1}^n(W_i-1)^2\sum_{i=1}^n \big(z^\top \bar \psi_0(X_i) - \bP_n [ z^\top \bar{\psi}_0(\cdot)] \big)^2 \Big]^{-1/2} 
    \sqrt{n} \sum_{i=1}^n (W_{R_i}-1)z^\top \bar \psi_0(X_i).
\end{align*}
Since $W$ is exchangeable, $H_n$ and $H_n^R$ have the same distribution and it remains to show
\begin{align*}
    \sup_{t \in \bR} \Big\lvert \P_{RW|X}\big(H_n^R \le t \big) - \P\big(N(0, 1) \le t\big) \Big\rvert = o_{\P_X}(1), \yestag \label{eq:matching,ANormalZ_R}
\end{align*}
or equivalently, any subsequence $\{n_k\}_{k \ge 0}$ contains a further subsequence $\{n_{k(\ell)}\}_{\ell \ge 0}$ such that
\begin{align*}
    \sup_{t \in \bR} \Big\lvert \P_{RW|X}\big(H_{n_{k(\ell)}}^R \le t \big) - \P\big(N(0, 1) \le t\big) \Big\rvert \stackrel{\sf a.s.}{\longrightarrow}  0
\end{align*}
almost surely as $\ell \to \infty$. For any $i,j \in \zahl{n}$, we define
\begin{align*}
    U_i^2 &:= \frac{(z^\top \bar \psi_0(X_i) - \Pn [z^\top \bar \psi_0(\cdot)] )^2}{\sum_{i=1}^n(z^\top \bar \psi_0(X_i) - \Pn [z^\top \bar \psi_0(\cdot)])^2}, \\
    V_i^2 &:= \frac{(W_{i}-1)^2}{\sum_{i=1}^n(W_{i}-1)^2}, \\
    \delta_{ij}^2 &:= 
    \frac{n (z^\top \bar \psi_0(X_i) - \Pn [z^\top \bar \psi_0(\cdot)] )^2 (W_{j}-1)^2}{\sum_{i=1}^n(z^\top \bar \psi_0(X_i) - \Pn [z^\top \bar \psi_0(\cdot)] )^2 \sum_{j=1}^n(W_{j}-1)^2} =n U_i^2 V_j^2,
\end{align*} 
and for any $\delta>0$,
\begin{align*}
    d_n(\delta) := \frac{1}{n} \sum_{i=1}^n \sum_{j=1}^n  \delta_{ij}^2 \ind\big(\delta_{ij}^2 > \delta \big).
\end{align*}
We then claim 
\begin{align*}
    \max_{i \in \zahl{n}} U_i^2 = o_{\P_X}(1) ~\text{ and }~  \max_{i \in \zahl{n}} V_i^2 = o_{\P_W}(1). \yestag \label{eq:matching,UV_Pcvrg}
\end{align*}
The first part in \eqref{eq:matching,UV_Pcvrg} follows from $\E[\lvert z^\top \bar \psi_0(X_i)\rvert^q] < \infty$, which is implied by Assumption~\ref{asp: BT,DML,1}(iii) and Assumption~\ref{asp: BT,DML,2}(ii). 
We can see this claim by observing that, for any $\epsilon>0$,
\begin{align*}
\P\Big(\max_{i \in \zahl{n}}\lvert z^\top \bar \psi_0(X_i) \rvert >\epsilon\sqrt n\Big)
\le n\P\big(\lvert z^\top \bar \psi_0(X_1)\rvert>\epsilon\sqrt n\big)
\le \epsilon^{-q} n^{1-q/2} \E\big[|z^\top \bar \psi_0(X_1)|^q\big], 
\end{align*}
and the upper bound holds as
\begin{align*}
    U_i^2=  \frac{(z^\top\bar \psi_0(X_i) - \Pn [z^\top \bar \psi_0(\cdot)])^2}{n \Pn (z^\top \bar \psi_0(\cdot) - \Pn [z^\top \bar \psi_0(\cdot)])^2} \le \frac{\max_{i \in \zahl{n}} 2  \lvert z^\top\bar \psi_0(X_i) \rvert^2}{n \Pn (z^\top \bar \psi_0(\cdot) - \Pn [z^\top \bar \psi_0(\cdot)])^2}
    \yestag \label{eq:DML,cvg,Ubound}
\end{align*}
with 
\[
\Pn (z^\top \bar \psi_0(\cdot) - \Pn [z^\top \bar \psi_0(\cdot)])^2 = z^\top\Sigma^2 z + o_{\P}(1). 
\]
For establishing the second part in \eqref{eq:matching,UV_Pcvrg}, we first note the uniform square-integrablility of $W$ in \eqref{eq:DML,BTweights,USI} and by Markov inequality, for any $\epsilon>0$, we can choose $n$ sufficiently large such that 
\begin{align*}
    \limsup_{n \to \infty} u^{2} \P\big( W_{i} >u \big) \le \limsup_{n \to \infty} \E \big[ W_{i}^2 \ind(W_{i} >u ) \big] \le \epsilon^2
\end{align*} 
for all $u \ge \epsilon \sqrt{m}$; by the first part of Proposition~\ref{Prop:BootstrapRate}, it then implies
\begin{align*}
    \frac{1}{\sqrt{n}} \E\Big[ \max_{i \in \zahl{n}} W_{i} \Big] 
    \le \epsilon + \sqrt{n} \Big[ \sup_{ u \ge \epsilon \sqrt{n}} u^{2} \P(W_{i} >u ) \Big]  \int_{\epsilon \sqrt{n}}^{\infty} u^{-2} \d u \le 2\epsilon.
\end{align*}
Following an analogous argument as \eqref{eq:DML,cvg,Ubound} above and applying Markov's inequality, one can show that the second part of \eqref{eq:matching,UV_Pcvrg} is achieved.
Next, we shall show that for all $\delta>0$, 
\[
d_n(\delta) = o_{\P_{XW}}(1). 
\]
For any $\epsilon>0$, let 
\begin{align*}
    A_n := \Big\{\max_{i \in \zahl{n}} U_i^2 < \epsilon \Big\} ~\text{ and }~ B_n := \Big\{ \frac{1}{n}\sum_{i \in \zahl{n}} (W_{i}-1)^2 > \frac{c^2}{2} \Big\}
\end{align*}
be the events with $\P_{XW}$-probability approaching one. It then follows that 
\begin{align*}
    \{\delta_{ij}^2 > \delta \} \cap A_n \cap B_n \subseteq \{n V_j^2 > \epsilon^{-1} \delta\} \cap B_n \subseteq \{(W_{j}-1)^2 > \epsilon^{-1} \delta c^2/2\},
\end{align*}
which gives
\begin{align*}
    \frac{1}{n} \sum_{i=1}^n \sum_{j=1}^n  \delta_{ij}^2 \ind\big(\delta_{ij}^2 > \delta \big) \ind_{A_n\cap B_n} &
    \le \sum_{i=1}^n \sum_{j=1}^n  U_i^2 V_j^2 \ind\big((W_{j}-1)^2 > \epsilon^{-1} \delta c^2 /2\big) \ind_{B_n} \\
    &= \frac{{n}^{-1} \sum_{j=1}^n (W_{j}-1)^2 \ind\big((W_{j}-1)^2 > \epsilon^{-1} \delta c^2  /2\big) \ind_{B_n}}{{n}^{-1} \sum_{j=1}^n (W_{j}-1)^2 }  \\
    & \le \frac{2}{c^2 n} \sum_{j=1}^n (W_{j}-1)^2 \ind\big((W_{j}-1)^2 > \epsilon^{-1} \delta c^2 /2\big).
\end{align*}
The above further implies that, for any $\epsilon'>0$, one has by Markov inequality
\begin{align*}
    & \P_{XW} \Big( \frac{1}{n} \sum_{i=1}^n \sum_{j=1}^n  \delta_{ij}^2 \ind\big(\delta_{ij}^2 > \delta \big) > \epsilon' \Big) \\
    & ~\le \P(A_n^c) + \P(B_n^c) + \P \Big( \frac{2}{c^2 n} \sum_{j=1}^n (W_{j}-1)^2 \ind\big((W_{j}-1)^2 > \epsilon^{-1} \delta c^2/2\big) > \epsilon' \Big) \\
    & ~\le \P(A_n^c) + \P(B_n^c) + \frac{2}{c^2 \epsilon'} \E \big[ (W_{j}-1)^2 \ind\big((W_{j}-1)^2 > \epsilon^{-1} \delta c^2 /2\big) \big].
\end{align*}
Due to the arbitrariness of $\epsilon$ and noting that \eqref{eq:DML,BTweights,USI} gives $W_{j}-1$ is uniformly square-integrable, we conclude
\begin{align*}
    \frac{1}{n} \sum_{i=1}^n \sum_{j=1}^n  \delta_{ij}^2 \ind\big(\delta_{ij}^2 > \delta \big) = o_{\P_{XW}}(1). \yestag \label{eq:matching,d_Pcvrg}
\end{align*}
Now, combining \eqref{eq:matching,UV_Pcvrg} and  \eqref{eq:matching,d_Pcvrg}, for any subsequence $\{n_k\}_{k \ge 1}$, we then have a further subsequence $\{n_{k(\ell)}\}_{\ell \ge 1}$ such that for all $\delta>0$,
\begin{align*}
    \max_{1\le i \le n_{k(\ell)}} U_i^2 \stackrel{\sf a.s.}{\longrightarrow} 0 ,~~  \max_{1\le i \le n_{k(\ell)}} V_i^2 \stackrel{\sf a.s.}{\longrightarrow} 0 ~\text{ and }~ d_{n_{k(\ell)}}(\delta) \stackrel{\sf a.s.}{\longrightarrow} 0, \yestag \label{eq:matching,UV_ascvrg}
\end{align*}
where the last claim follows from a diagonal argument and noting that $d_n(\delta)$ is decreasing in $\delta$. Conditioning on $(X,W)$ and applying Theorem 4.1 in \cite{hajek1961some}, the linear rank statistics satisfies
\begin{align*}
    \sup_{t \in \bR} \Big\lvert \P_{R| XW}\big(H_{n_{k(\ell)}}^{R} \le t \big) - \P\big(N(0,1) \le t\big) \Big\rvert \stackrel{\sf a.s.}{\longrightarrow} 0
\end{align*}
as $\ell \to \infty$. A conditional bounded convergence theorem then yields 
\begin{align*}
    \sup_{t \in \bR} \Big\lvert \P_{RW|X}\big(H_{n_{k(\ell)}}^{R} \le t \big) - \P\big(N(0, 1) \le t\big) \Big\rvert \stackrel{\sf a.s.}{\longrightarrow} 0
\end{align*}
as $\ell \to \infty$, which completes the proof of \eqref{eq:matching,ANormalZ_R}. By Lemma~\ref{lemma:StoOrders}, this yields the row-wise version of \eqref{eq:DML,BTconsistency1} along the arbitrary sequence $\{\Q_n\}_{n\ge1}$. Since the latter sequence was arbitrary, Lemma~\ref{lemma:UniformSequenceSO} upgrades this row-wise convergence to the uniform statement in \eqref{eq:DML,BTconsistency1}.
\end{proof}

Lastly, we give the proof for Proposition~\ref{Prop:BootstrapRate} that characterizes the behavior of the rate induced by the bootstrap procedure. Specially, for Efron's bootstrap, $\max_{i\in \zahl{n}} W_i$ is the maximum occupancy in the classical $n$-balls-$n$-bins model. In particular, it is classical that this quantity is of order $\log n / \log\log n$; cf. \cite{gonnet1981expected}'s analysis of separate chaining and the discussion in \cite{raab1998balls}. We give a direct proof here for completeness.

\begin{proof}[Proof of Proposition~\ref{Prop:BootstrapRate}]
We prove the proposition in four parts.

\vspace{0.2cm}

{\bf Part I. Proving \eqref{eq:generate-rate-bootstrap}.}
For the lower bound, note $\max_{i \in \zahl{n}}W_i \ge \frac{1}{n}\sum_{i=1}^nW_i = 1$. For the upper bound, write
\begin{align*}
    \frac{1}{\sqrt{n}} \E\Big[ \max_{i \in \zahl{n}} W_{i} \Big] & = \frac{1}{\sqrt{n}} \int_{0}^{x} \P\Big( \max_{i \in\zahl{n}} W_{i} >t \Big) \d t + \frac{1}{\sqrt{n}} \int_{x}^{\infty} \sum_{i \in \zahl{n}} \P( W_{i} >t ) \d t \\
    & \le \frac{x}{\sqrt{n}} + \sqrt{n} \Big[ \sup_{ t \ge x} t^{2} \P(W_{i} >t ) \Big]  \int_{x}^{\infty} t^{-2} \d t, 
\end{align*}
and the corresponding rates in (a)-(c) can be readily obtained.

\vspace{0.2cm}

{\bf Part II. Proving Claim (i).}
Now we give the rate for Efron's bootstrap. First, we derive the upper bound for $\E[\max_{i\in\zahl{n}}W_i]$.  For any $\epsilon>0$ with $L_n := \log n / \log\log n$, define 
\[
m_+ := (1+\epsilon)L_n. 
\]
We write
\begin{align*}
    \E\Big[\max_{i \in \zahl{n}}W_i\Big] = \sum_{m < m_+}\P \Big(\max_{i \in \zahl{n}}W_i \geq m\Big) + \sum_{m \ge m_+} \P \Big(\max_{i \in \zahl{n}}W_i \geq m\Big) \leq m_+ + \sum_{m \ge m_+} \P \Big(\max_{i \in \zahl{n}}W_i \geq m\Big).
\end{align*}
By the arbitrariness of $\epsilon$, it only suffices to show the second term goes to $0$. By Chernoff bound, for any $m>0$ and setting $\lambda = \log(m)$ gives
\begin{align*}
    \P(W_1 \ge m) \le \exp(-\lambda m)\Big(1-\frac{1}{n}+\frac{1}{n}\exp(\lambda)\Big)^n \le \exp\big\{\exp(\lambda)-1-\lambda m \big\} = e^{-1}\Big(\frac{e}{m}\Big)^m,
\end{align*}
which implies
\begin{align*}
    \P \Big(\max_{i \in \zahl{n}}W_i \geq m\Big) \le ne^{-1}\Big(\frac{e}{m}\Big)^m := B(m).
\end{align*}
Notice for large $m$,
\begin{align*}
    \frac{B(m+1)}{B(m)} = \frac{e}{m+1}\Big(\frac{m}{m+1}\Big)^m \asymp \frac{1}{m};
\end{align*}
then 
\begin{align*}
    \sum_{m \ge m_+} \P \Big(\max_{i \in \zahl{n}}W_i \geq m\Big) \lesssim B(m_+) \sum_{j\geq0} \Big( \frac{1}{m_+} \Big)^j  \lesssim B(m_+)
\end{align*}
and the desired result follows immediately from 
\[
\log(B(m_+)) = -\epsilon \log(n) + o(\log(n)) \to -\infty
\]
 as $n$ goes to infinity. 
 
 Next we derive the lower bound. With any $\epsilon>0$ and the same $L_n$ we write 
 \[
 m_-:= \floor{(1-\epsilon)L_n}{}. 
 \]
 By Markov's inequality,
\begin{align*}
    \E\Big[\max_{i \in \zahl{n}}W_i\Big] \ge m_- \P \Big(\max_{i \in \zahl{n}}W_i \geq m_-\Big) \ge m_- \P(N_{m_-}>0) \ge m_- \frac{\E[N_{m_-}]^2}{\E[N_{m_-}^2]},
\end{align*}
where for any $m>0$, we define $N_m$ as the number of weights such that the weight takes the value $m$, that is, $ N_m := \sum_{i=1}^n \ind(W_i = m)$. Moreover, the last inequality above is obtained by applying the Paley–Zygmund inequality. Then by the arbitrariness of $\epsilon$, it only suffices to show 
\[
\E[N_{m_-}]^2/\E[N_{m_-}^2]= 1+o(1). 
\]

For the first moment, we have when $m = o(n^{1/2})$,
\begin{align*}
    \E[N_m] &= n\P(W_1=m) = n \binom{n}{m}n^{-m}\Big(1-\frac{1}{n} \Big)^{n-m} \\
    & = n (1-o(1))\frac{1}{m!}(e^{-1}+o(1))(1+o(1))= n \frac{e^{-1}+o(1)}{m!},
\end{align*}
where the first equality on the second line comes from 
\begin{align*}
    \frac{1}{m!} \ge \binom{n}{m}n^{-m} \ge \frac{1}{m!} \Big(1- \frac{m(m-1)}{2n}\Big) = \frac{1-o(1)}{m!} 
\end{align*}
and
\begin{align*}
    \log\Big\{\Big(1-\frac{1}{n}\Big)^{-m}\Big\} = -\frac{m}{n}+O\Big(\frac{1}{n^2}\Big) \to 0.
\end{align*}
Then by Stirling's formula, 
\begin{align*}
    \log(\E[N_{m_-}]) = \log(n) - m_-\log(m_-) -m_- +O(\log(m_-)) = \epsilon\log(n) + o(\log(n)),
\end{align*}
which implies $\E[N_{m_-}] \to \infty$. 

For the second moment, for any $m=o(n^{1/2})$, with a similar analysis we have 
\begin{align*}
    \E[N_{m}^2] &= \E[N_{m}] + n(n-1)\P(W_1=W_2=m) \\
    &= \E[N_{m}] + n(n-1) \frac{n!}{m!m!(n-2m)!}n^{-2m}\Big(1-\frac{2}{n}\Big)^{n-2m}\\   
    &= \E[N_{m}] + n(n-1) \frac{e^{-2}+o(1)}{(m!)^2}.
\end{align*}
Comparing this with the order for $\E[N_{m_-}]$, we derive the desired result.

\vspace{0.2cm}

{\bf Part III. Proving Claim (ii).} We first prove the rate for normalized multiplier bootstrap with exponential-tail weights. Denoting $\E[Y_1]$ as $\mu$, we observe that for any $s>0$, by Chernoff bound, one has
\begin{align*}
    \E\Big[\max_{i \in \zahl{n}}W_i\Big] 
    \le \frac{2}{\mu} \E\Big[\max_{i \in \zahl{n}}Y_i\Big] + n\P\Big(\frac{1}{n}\sum_{i=1}^n Y_i < \frac{\mu}{2}\Big) 
    \le \frac{2}{\mu} \E\Big[\max_{i \in \zahl{n}}Y_i\Big] + n e^{sn\mu/2}\big(\E[e^{-sY_1}]\big)^n.
\end{align*}
Further note that, by sub-exponentiality, the second term above is exponentially small as 
\[
\E[e^{-sY_1}]<e^{-s\mu/2}~~ \text{ for sufficiently small $s$}.
\]
Additionally, the first term satisfies
\begin{align*}
    \E\Big[\max_{i \in \zahl{n}}Y_i\Big] \le x + n \int_{x}^\infty \P(Y_1 \ge t)\d t \lesssim \log(n),
\end{align*}
which conclude the case for the normalized multiplier bootstrap. 

For the double bootstrap, note that, for any $\lambda>0$,
\begin{align*}
    \E[e^{\lambda W_1} | M_1] = \Big(1+\frac{M_1}{n}(e^\lambda-1)\Big)^n \le \exp\{(e^\lambda-1)M_1\},
\end{align*}
which implies 
\begin{align*}
    \E[e^{\lambda W_1}] \le \E[\exp\{(e^\lambda-1)M_1\}] \le \Big(1+\frac{e^{e^\lambda-1}-1}{n} \Big)^n \le \exp\{e^{e^\lambda-1}-1\}.
\end{align*}
By Chernoff bound, we know $W_1$ has a uniform exponential tail and the first part of the proof gives the claimed result.

\vspace{0.2cm}

{\bf Part IV. Proving Claim (iii).} 
We prove the existence by considering the grouped delete-$h$ jackknife, which has been discussed in Example 3.6 of \cite{praestgaard1993exchangeably} and satisfies Assumption~\ref{asp: BTWeights}. In this bootstrap scheme, we set 
\[
\cS_{n,h_n}:=\Big\{S\subseteq \zahl{n}: \lvert S\rvert=h_n\Big\} 
\]
and let $S_n$ be drawn uniformly from $\cS_{n,h_n}$. For each $i \in \zahl{n}$, the bootstrap weight component is defined as 
\[
W_i := \frac{n}{n-h_n}\ind(i\notin S_n). 
\]
Note for this bootstrap procedure, if $h_n/n \to \gamma \in (0,1)$, then $c^2= \lim_{n\to \infty} \frac{h_n}{n-h_n} =\frac{\gamma}{1-\gamma}$ and thus
\begin{align*}
a_n = \frac{1}{\sqrt{n}}\frac{n}{n-h_n} \asymp n^{-1/2}.
\end{align*}
This completes the proof.
\end{proof}

\subsection{Auxiliary lemmas}

\begin{lemma}[Multiplier inequality, \cite{wellner1996bootstrapping}]\label{lemma:MultIneq}

Let $W = (W_1, \ldots, W_n)^{\top}$ be a non-negative exchangeable random vector on $(\Omega_W, \cA_W, \P_W)$ and, for every $n$,
\begin{align*}
    S_n := \int_0^{\infty} \sqrt{\P_W \left(W_1 \ge t\right)} \d t < \infty.
\end{align*}
Let $\{Z_i\}_{i=1}^n$'s be i.i.d. $\ell^{\infty}(\mathcal{F}_n)$-valued random elements in $(\Omega_X^{\infty}, \cA_X^{\infty}, \P_X^{\infty})$. Write $\lVert \cdot \rVert_n = \sup_{f \in \mathcal{F}_n} \lvert Z_i(f) \rvert$ and assume $\{Z_i\}_{i=1}^n$'s are independent with $W$. Then for any $n_0, n$ such that $1 \le n_0 < n < \infty$, we have the following inequality:
\begin{align*}
    \E_{XW}^o \Big[ \frac{1}{\sqrt{n}} \Big\lVert \sum_{i=1}^n W_iZ_i \Big\rVert_n \Big] &\le n_0 \E_{XW}^o \big[ \big\lVert Z_1 \big\rVert_n \big] \frac{\E_{W}\big[ \max_{1\le i \le n} W_i \big]}{\sqrt{n}} \\
    & ~~+ S_n \max_{n_0 < i \le n} \E_{X}^o \Big[ \frac{1}{\sqrt{i}} \Big\lVert \sum_{j=n_0+1}^i Z_j \Big\rVert_n \Big].
\end{align*}
\end{lemma}

Note an immediate consequence of Lemma~\ref{lemma:MultIneq} is that, for any deterministic subset $I$ of $\zahl{n}$ with $\lvert I \rvert = m$, the same subset version result holds with $(W_i)_{i=1}^n$ replaced by $(W_i)_{i\in I}$, summation of $n$ terms replaced by $m$ terms, and $n^{-1/2}\E[\max_{1\le i \le n}W_i]$ replaced by $m^{-1/2}\E[\max_{i\in I}W_i]$.



\begin{lemma}[Maximal Inequality, \cite{chernozhukov2014gaussian}]
\label{lemma:MaximalIneq} 
Let $\{X_i\}_{i=1}^n$ be i.i.d. random elements in $(\Omega_X, \cA_X, \P_X )$. Suppose that $F\geq \sup_{f \in \mathcal{F}}\lvert f\rvert$ is a measurable envelope for a suitably measurable function class $\mathcal{F}$ 
with $\lVert F\rVert_{\P,q} < \infty$ for some $q \geq 2$.  Let $M := \max_{1\le i\le n} F(X_i)$ and $\sigma^{2} > 0$ be any positive constant such that $\sup_{f \in \mathcal{F}}  \lVert f \rVert_{\P,2}^{2} \leq \sigma^{2} \leq \lVert F \rVert_{\P,2}^{2}$. Suppose that there exist some constants $a \ge e$ and $v \ge 1$ such that
\begin{align*}
\log \sup_{\Q} N\big(\epsilon \| F \|_{\Q,2}, \mathcal{F},  \| \cdot \|_{\Q,2}\big) \leq  v \log (a/\epsilon), \ 0 <  \epsilon \leq 1.
\end{align*}
Then
\begin{align*}
\E_\P \big[ \lVert \bG_{n} \rVert_{\mathcal{F}} \big] 
\leq K  \bigg\{ \sqrt{v\sigma^{2} \log \Big( \frac{a \lVert F \rVert_{\P,2}}{\sigma} \Big) } 
+ \frac{v\lVert M \rVert_{\P, 2}}{\sqrt{n}} \log \Big( \frac{a \lVert F \rVert_{\P,2}}{\sigma} \Big) \bigg\},
\end{align*}
where $K$ is an absolute constant.  Moreover, for every $t \geq 1$, with probability $> 1-t^{-q/2}$,
\begin{align*}
\lVert \bG_{n} \rVert_{\mathcal{F}} \leq (1+\alpha) \E_\P \big[ \lVert \bG_{n} \rVert_{\mathcal{F}} \big] 
+ K(q) \Big\{ (\sigma + n^{-1/2} \lVert M \rVert_{\P,q}) \sqrt{t}
+  \alpha^{-1}  n^{-1/2} \lVert M \rVert_{\P,2}t \Big\}, 
\ \forall \alpha > 0,
\end{align*}
where $K(q) > 0$ is a constant depending only on $q$.  In particular, setting $a \geq n$ and $t = \log n$,
with probability $> 1- c(\log n)^{-1}$,
\begin{align*}
\lVert \bG_{n} \rVert_{\mathcal{F}} \leq K(q,c) \bigg\{ \sigma \sqrt{v \log \Big ( \frac{a \lVert F \rVert_{\P,2}}{\sigma} \Big ) } + \frac{v
 \| M \|_{\P,q} } {\sqrt{n}}\log \Big( \frac{a \lVert F \rVert_{\P,2}}{\sigma} \Big) \bigg\},
\end{align*} 
where $  \lVert M \rVert_{\P,q}  \leq n^{1/q} \lVert F\rVert_{\P,q}$ and  $K(q,c) > 0$ is a constant depending only on $q$ and $c$.
\end{lemma}

\begin{lemma}[Sequence characterization of uniform probabilistic statements]
\label{lemma:UniformSequenceSO}
Let $\{\cP_n\}_{n\ge1}$ be arbitrary classes of probability measures, let $\{b_n\}_{n\ge1}$ be a sequence of positive deterministic numbers, and for each $n$ and each $\P \in \cP_n$, let $Z_{n,\P}$ be a random element defined on a probability space with probability law denoted by $\P_{\P}$.
Then the following conditions are equivalent:
    \begin{enumerate}[itemsep=-.5ex, label=(\alph*)]
        \item for every $\epsilon>0$, $\sup_{\P\in\cP_n}\P_{\P}\big(\lVert Z_{n,\P} \rVert>\epsilon b_n\big)\to 0$;
        \item for every $\epsilon>0$ and every sequence $\{\P_n\}_{n\ge1}$ such that $\P_n\in\cP_n$ for all $n$, $\P_{\P_n}\big(\lVert Z_{n,\P_n} \rVert>\epsilon b_n\big)\to 0.$
    \end{enumerate}
Moreover, the following conditions are equivalent:
    \begin{enumerate}[itemsep=-.5ex, label=(\alph*)]
        \item for every deterministic sequence $\ell_n\to\infty$, $\sup_{\P\in\cP_n}\P_{\P}\big(\lVert Z_{n,\P} \rVert>\ell_n b_n\big)\to 0$;
        \item for every deterministic sequence $\ell_n\to\infty$ and every sequence $\{\P_n\}_{n\ge1}$ such that $\P_n\in\cP_n$ for all $n$, $\P_{\P_n}\big(\lVert Z_{n,\P_n}\rVert>\ell_n b_n\big)\to 0.$
    \end{enumerate}
The same equivalences remain valid when $\P_{\P}$ is replaced throughout by an outer probability measure.
\end{lemma}

\begin{proof}
{\bf Part I.} We first prove the sequence and uniform characterizations for convergence to zero in probability are equivalent. It is trivial that (a) implies (b). Conversely, suppose (a) fails, then there exist $\epsilon_0>0$, $\eta_0>0$, and a subsequence $\{n_k\}_{k\ge1}$ such that
\begin{align*}
\sup_{\P\in\cP_{n_k}}\P_{\P}\big(\lVert Z_{n_k,\P}\rVert>\epsilon_0 b_{n_k}\big)>\eta_0
\end{align*}
for all $k$. Hence one can choose $\P_{n_k}\in\cP_{n_k}$ such that
\begin{align*} \P_{\P_{n_k}}\big(\lVert Z_{n_k,\P_{n_k}}\rVert>\epsilon_0 b_{n_k}\big)>\eta_0/2
\end{align*}
for all $k$, Extending $\{\P_{n_k}\}_{k\ge1}$ arbitrarily to a full sequence $\{\P_n\}_{n\ge1}$ with $\P_n\in\cP_n$ contradicts (b).

\vspace{0.2cm}

{\bf Part II.} Similarly, to prove the sequence and uniform characterizations of boundedness in probability are equivalent, we note (a) implies (b) from
\begin{align*}
\P_{\P_n}\big(\|Z_{n,\P_n}\|>\ell_n b_n\big)
\le    \sup_{\P\in\cP_n}\P_{\P}\big(\|Z_{n,\P}\|>\ell_n b_n\big).
\end{align*}
Conversely, if (a) fails, there then exist a deterministic sequence $\ell_n\to\infty$, a constant $\eta_0>0$, and a subsequence $\{n_k\}_{k\ge1}$ such that
\begin{align*}
    \sup_{\P\in\cP_{n_k}}\P_{\P}\big(\lVert Z_{n_k,\P}\rVert>\ell_{n_k} b_{n_k}\big)>\eta_0
\end{align*}
for all $k$. Hence one can choose $\P_{n_k}\in\cP_{n_k}$ such that
\begin{align*}
    \P_{\P_{n_k}}\big(\lVert Z_{n_k,\P_{n_k}}\rVert>\ell_{n_k} b_{n_k}\big)>\eta_0/2
\end{align*}
for all $k$. Extending $\{\P_{n_k}\}_{k\ge1}$ arbitrarily to $\{\P_n\}_{n\ge1}$ with $\P_n\in\cP_n$ contradicts (b).

\vspace{0.2cm}
{\bf Part III.} The last claim for outer probabilities follows by the same proof, with $\P_{\P}$ replaced throughout by the corresponding outer probability.
\end{proof}

Given a real-valued random quantity $V_n$ defined in the product probability space \eqref{eq:ProbSpace}, we say $V_n$ is of an order $o_{\P_W}^o(1)$ in $\P_X^o$-probability if for any $\epsilon, \delta>0$
\begin{align*}
    \P_X^o \big( \P_{W|X}^o(\lvert V_n \rvert > \epsilon) > \delta \big) \to 0 ~~~~~\text{as } n \to \infty.
\end{align*}
The translation of stochastic orders, given in Lemma \ref{lemma:StoOrders}, shows that such quantites do not have an impact on the bootstrap distribution, which plays an important role in our proof. Suppose that some measurability conditions are satisfied such that the Fubini's theorem in the following equation~\eqref{eq:DML,Lemma,Meas} can be used freely. We then have the following claim holds true.

\begin{lemma}[Transition of stochastic orders, modified from Lemma 3 in \cite{cheng2010bootstrap}]\label{lemma:StoOrders}
  For a random quantity $V_n$, 
   we have 
  \[
  V_n = o_{\P_{XW}}^o(1) \text{ if and only if }V_n=o_{\P_W}^o(1) \text{ in $\P_{X}^o$-probability.}  
  \]
  Moreover, when a random variable $Q_n$ on the product probability space has a bootstrap distribution that asymptotically imitates the unconditional distribution of $Q$, which has a continuous cumulative distribution function, that is,
  \begin{align*}
      \sup_{t \in \bR} \big\lvert \P_{W|X}\big(Q_n \le t \big) - \P(Q \le t) \big\rvert &= o_{\P_{X}}^o(1),
  \end{align*}
  we then also have the same claim holds true for $Q_n + V_n$, if $V_n=o_{\P_W}^o(1)$ in $\P_{X}^o$-probability:
  \begin{align*}
       \sup_{t \in \bR} \big\lvert \P_{W|X}\big(Q_n+V_n \le t \big) - \P(Q \le t) \big\rvert &= o_{\P_{X}}^o(1). \yestag \label{eq:StoOrd,op1,Slutsky}
  \end{align*}
\end{lemma}

\begin{proof}

{\bf Part I.}
    We first prove 
    \begin{align*}
        V_n = o_{\P_{XW}}^o(1) \text{ if and only if } V_n=o_{\P_W}^o(1) \text{ in $\P_{X}^o$-probability}.
    \end{align*}
    For all $\epsilon,\delta>0$, by Markov's inequality one has
    \begin{align*}
        \P_{X}^o(  \P_{W|X}^o(\lvert V_n \rvert \ge \epsilon) \ge \delta ) \le \frac{1}{\delta} \E_{X}^o\big[\P_{W|X}^o(\lvert V_n \rvert \ge \epsilon) \big] = \frac{1}{\delta} \E_{X}^o\big[\E_{W|X}^o\{\ind(\lvert V_n \rvert \ge \epsilon)\} \big].
    \end{align*}
    From Fubini's theorem for repeated outer expectations \citep[Lemma 6.14]{kosorok2008introduction}, we can further bound the righthand side of the above by $$\E_{XW}^o [\ind(\lvert V_n \rvert \ge \epsilon)] = \P_{XW}^o(\lvert V_n \rvert \ge \epsilon).$$ Therefore we conclude that $V_n = o_{\P_{XW}}^o(1)$ implies $V_n=o_{\P_W}^o(1)$ in $\P_{X}^o$-probability. 
    
    The other direction follows from for any $\epsilon$ and an arbitrary $\eta$,
    \begin{align*}
        \P_{XW}^o(\lvert V_n \rvert \ge \epsilon) &= \E_{X}^o\big[\P_{W|X}^o(\lvert V_n \rvert \ge \epsilon) \big]
        \yestag \label{eq:DML,Lemma,Meas} \\ 
        &\le \E_{X}^o\big[\P_{W|X}^o(\lvert V_n \rvert \ge \epsilon) \ind \big(\P_{W|X}^o(\lvert V_n \rvert \ge \epsilon) \ge \eta \big) \big] \\
        & ~ + \E_{X}^o\big[\P_{W|X}^o(\lvert V_n \rvert \ge \epsilon) \ind \big(\P_{W|X}^o(\lvert V_n \rvert \ge \epsilon) < \eta \big) \big] \\
        &= \E_{X}^o\big[ \ind \big(\P_{W|X}^o(\lvert V_n \rvert \ge \epsilon) \ge \eta \big) \big] + \eta \\
        &\le \P_{X}^o\big( \P_{W|X}^o(\lvert V_n \rvert \ge \epsilon) \ge \eta \big) + \eta.
    \end{align*}
    Specially, we note that any random element $V_n$ defined only on $( \Omega_X^\infty, \cA_X^\infty, \P_X^\infty )$ with stochastic order $o_{\P_X}^o(1)$ is also of of an order $o_{\P_W}^o(1)$ in $\P_X^o$-probability. As for any $\epsilon, \delta>0$, we have
    \begin{align*}
        \P_X^o \big( \P_{W|X}^o(\lvert V_n \rvert > \epsilon) > \delta \big) &\le \frac{1}{\delta} \E_X^o \big[ \E_{W|X}^o\{\ind(\lvert V_n \rvert > \epsilon) \}\big] = \frac{1}{\delta} \E_X^o \big[ \ind(\lvert V_n \rvert > \epsilon)\big] =\frac{1}{\delta} \P_X^o (\lvert V_n \rvert > \epsilon),
    \end{align*}
    by Markov's inequality and noting $V_n$ does not depend on the bootstrap weights.

\vspace{0.2cm}

    {\bf Part II.}
    Now we proceed to prove that $V_n$ does not have an impact on the bootstrap distribution of $Q_n$. For any $t \in \bR$ and $\epsilon>0$, on one hand, we have 
    \begin{align*}
         \P_{W|X}^o(Q_n+V_n \le t) &= \P_{W|X}^o(Q_n+V_n \le t, \lvert V_n \rvert \ge \epsilon) + \P_{W|X}^o(Q_n+V_n \le t, \lvert V_n \rvert < \epsilon) \\
         &\le \P_{W|X}^o(\lvert V_n \rvert \ge \epsilon) + \P_{W|X}^o(Q_n \le t + \epsilon).
    \end{align*}
    On the other hand, we have
    \begin{align*}
        \P_{W|X}^o(Q_n+V_n \le t) &= 1-\P_{W|X}^i(Q_n+V_n> t) \\
        &=  1-\P_{W|X}^i(Q_n+V_n> t, \lvert V_n \rvert < \epsilon) -\P_{W|X}^i(Q_n+V_n> t, \lvert V_n \rvert \ge \epsilon) \\
        & \ge 1-\P_{W|X}^i(Q_n> t-\epsilon) -\P_{W|X}^i( \lvert V_n \rvert \ge \epsilon),
    \end{align*}
    where $\P_{W|X}^i$ stands for the corresponding inner probability. Combining these two inequalities, we obtain that, for any $t \in \bR$ and $\epsilon>0$, 
    \begin{align*}
        o_{\P_{X}}^o(1) + \P_{W|X}^o(Q_n \le t-\epsilon) \le \P_{W|X}^o(Q_n+V_n \le t) \le \P_{W|X}^o(Q_n \le t + \epsilon) + o_{\P_{X}}^o(1).
    \end{align*}
    Accordingly, at all $t$ such that $ \P(Q \le t)$ is continuous at $t$, we can take $\epsilon \downarrow 0$ and establish 
    \[
    \P_{W|X}^o(Q_n+V_n \le t) = \P(Q \le t) + o_{\P_{X}}^o(1). 
    \]
    The uniform result in \eqref{eq:StoOrd,op1,Slutsky} follows from Polya's theorem.
\end{proof}


{
\bibliographystyle{apalike}
\bibliography{AMS}
}

\end{document}